\newif\ifpictures
\picturestrue

\documentclass[12pt]{amsart}

\usepackage{a4,latexsym,amsfonts,amsmath,color,graphicx}
\usepackage{psfrag}

\DeclareMathOperator{\val}{val}
\DeclareMathOperator{\T}{\mathcal{T}}   
\DeclareMathOperator{\V}{\mathcal{V}}   
\DeclareMathOperator{\trop}{trop}       
\DeclareMathOperator{\new}{New}         
\DeclareMathOperator{\aff}{aff}
\DeclareMathOperator{\conv}{conv}
\DeclareMathOperator{\face}{face}

\DeclareMathOperator{\LC}{LC} 
\DeclareMathOperator{\contr}{\alpha} 
\DeclareMathOperator{\SIP}{P}

\begin{document}

\newtheorem{defi}{Definition}[section]
\newtheorem{satz}[defi]{Satz}
\newtheorem{lemma}[defi]{Lemma}
\newtheorem{kor}[defi]{Corollary}
\newtheorem{prop}[defi]{Proposition}
\newtheorem{bem}[defi]{Quotation}
\newtheorem{theo}[defi]{Theorem}
\newtheorem{conj}[defi]{Conjecture}
\newtheorem{bsp1}[defi]{Example}
\newtheorem{observation}[defi]{Observation}
\newenvironment{bsp}{\begin{bsp1}\rm}{\end{bsp1}}
\newtheorem{remark1}[defi]{Remark}
\newenvironment{remark}{\begin{remark1}\rm}{\end{remark1}}
\newcommand{\ZZ}{\mathbb{Z}}
\newcommand{\Z}{\mathbb{Z}}
\newcommand{\RR}{\mathbb{R}}
\newcommand{\CC}{\mathbb{C}}
\newcommand{\R}{\mathbb{R}}
\newcommand{\NN}{\mathbb{N}}
\newcommand{\C}{\mathbb{C}}
\newcommand{\QQ}{\mathbb{Q}}
\newcommand{\XX}{\mathcal{X}}
\newcommand{\beweis}{{\bfseries{Proof}\hspace{5mm}}}
\newcommand{\note}{{\bfseries{Notation}\hspace{5mm}}}
\newcommand{\atopfrac}[2]{\genfrac{}{}{0pt}{}{#1}{#2}}
\newcommand{\s}{{\ss\;}}

\title[Projections of tropical varieties and their self-intersections]{Projections of tropical varieties \\ and their self-intersections}

\author{Kerstin Hept and Thorsten Theobald}
\address{FB 12 -- Institut f\"ur Mathematik,
  Goethe-Universit\"at, 
  Postfach 111932, D-60054 Frankfurt am Main, Germany} 
\email{\{hept,theobald\}@math.uni-frankfurt.de}



\begin{abstract}
We study algebraic and combinatorial aspects of (classical)
projections of $m$-dimensional tropical varieties onto
$(m+1)$-dimensional planes. 
Building upon the work of Sturmfels, Tevelev, and
Yu on tropical elimination as well as the work of the authors 
on projection-based tropical bases, we characterize
algebraic properties of the relevant ideals and provide a 
characterization of the dual subdivision (as a subdivision
of a fiber polytope). This dual subdivision naturally leads to
the issue of self-intersections of a tropical variety under
projections. For the case of curves,
we provide some bounds for the (unweighted) number of 
self-intersections of projections onto the plane and give
constructions with many self-intersections. 
\end{abstract}

\maketitle


\section{Introduction\label{se:introduction}}

In the last years, tropical geometry has received much interest as
a field combining aspects of algebraic geometry, discrete geometry,
and computer algebra (for general background see, e.g., 
\cite{gathmann-2006, mikhalkin-icm,rgst}). The basic setup is
as follows.
Given a field $K$ with a real valuation $\val:K\rightarrow \mathbb{R}_{\infty}=\mathbb{R}\cup \{ \infty\}$ (i.e.\ $K=\mathbb{Q}$ with the $p$-adic valuation or 
the field $K=\mathbb{C}\{\{t\}\}$ of Puiseux series with the natural valuation)
the valuation map extends to an algebraic closure $\bar{K}$ 
and to $\bar{K}^n$ via
\[ \val \, : \, \bar{K}^n \ \rightarrow \ \mathbb{R}_{\infty}^n, \quad (a_1,\ldots, a_n) \ \mapsto\ (\val(a_1),\ldots,\val(a_n)) \, .
\]
Then for any polynomial $f=\sum_{\alpha} c_\alpha x^\alpha \in 
K[x_1,\ldots, x_n]$ the \emph{tropicalization of} $f$ is defined as
the polynomial over the tropical semiring $(\R_{\infty}, \oplus, \odot) := 
(\R_{\infty}, \min, +)$
\[\trop(f) \ =  \ \bigoplus_\alpha \val(c_\alpha)\odot x^\alpha \ = \ \min_\alpha\{\val(c_\alpha)+\alpha_1x_1+\cdots+\alpha_nx_n\} \, , \]
and the \emph{tropical hypersurface} of $f$ is 
\[\T(f) \ = \ \{w\in\mathbb{R}^n \, : \, \mbox{the minimum in $\trop(f)$ }
 \mbox{is attained at least twice in $w$}\} \, . 
\]
For an ideal $I \lhd K[x_1, \ldots, x_n]$,
the \emph{tropical variety of} $I$ is given by
$\T(I) \ = \ \bigcap_{f\in I}\T(f)$
or equivalently (if the valuation is nontrivial) by the topological closure
$\T(I)=\overline{\val \V(I)}$
where $\V(I)$ is the subvariety of $I$ in
$(\bar{K}^*)^n$.

In this paper, we continue the study of the following two lines of 
research and in particular advance their connections.
Firstly, \emph{tropical elimination} as developed in
\cite{sturmfels-tevelev-2008,sturmfels-tevelev-yu-2007, 
sturmfels-yu-2008}
(see also 
\cite{emiris-konaxis-palios,esterov-khovanski-2006}) can be seen as tropical
analog of elimination theory. 
Secondly, a natural way to handle an
ideal is by means of a basis, i.e., a finite set of generators.
In the stronger notion of
a \emph{tropical basis} it is additionally required that the set-theoretic
intersection of the tropical hypersurfaces of the generators
coincides with the tropical variety. That is,
a \emph{tropical basis} of the ideal $I$ is a finite generating set 
$\mathcal{F}$ of $I$, such that 
\[ \T(I) \ = \ \bigcap_{f\in \mathcal{F}}\T(f) \, . \] 
The systematic study of tropical bases has been initiated in 
\cite{bjsst,speyer-sturmfels-2004}. 
Recently, by revisiting the regular projection technique from
Bieri and Groves \cite{bg},
Hept and Theobald \cite{hept-theobald-2009} showed that
every tropical variety has a tropical short basis. If $I$ is an $m$-dimensional
prime ideal then for these bases, the projections of $\T(I)$
under a rational projection $\pi : \R^n \to \R^{m+1}$
play a key role (see Proposition~\ref{th:piinvpi}).

In the present paper,
our point of departure is the 
observation that geometrically, the projections 
$\pi(\T(I))$ and the fibers $\pi^{-1}(\pi(\T(I)))$
can be seen as an important special case of tropical elimination.
As a basic ingredient, we start by studying the bases from
the viewpoint of tropical elimination and its relations to
mixed fiber polytopes as recently developed by Sturmfels, Tevelev
and Yu
\cite{sturmfels-tevelev-2008,sturmfels-tevelev-yu-2007, 
sturmfels-yu-2008};
see also \cite{emiris-konaxis-palios,esterov-khovanski-2006}.
In these papers, it is shown that in various
situations the Newton polytope of the polynomial generating this
hypersurface is affinely isomorphic to a mixed fiber polytope.

The contributions of the current paper are as follows.
In Section~\ref{se:projelim}, we extend the study of projections 
of tropical varieties via elimination theory, thus providing
some useful tools for the subsequent sections. In detail,
we show that various properties of the starting ideal $I$ carry 
over to an auxiliary ideal on which elimination is actually applied
(see Lemma~\ref{le:carryover}).

In Section~\ref{se:combinatorics}, we refine the global viewpoint
on the connection of the Newton polytope and fiber polytopes
(as stated above) by studying as well the subdivision of the 
Newton polytope corresponding to $\pi^{-1} \pi \T(I)$. 
We mainly concentrate on the case of complete intersections.
To establish this characterization, we provide
some useful techniques to handle the affine isomorphisms connected
with the mixed fiber polytopes.
The local cells are sums of local mixed 
fiber polytopes (see Theorem~\ref{ko:subdivision}).
Since all these fiber polytopes are only determined up to affine
isomorphisms, we then show how to patchwork the local
mixed fiber polytopes
(Statements~\ref{theo:conv}--\ref{cor:vector}).

These considerations lead to the aspect of self-intersections of
images of tropical varieties under projections (as introduced in Section~\ref{se:selfinter}).
In Section~\ref{se:bounds} we study the (unweighted)
number of self-intersections of a 
tropical curve $\mathcal{C} \subseteq \R^n$ under a rational 
projection $\pi : \R^n \to \R^2$. We give constructions with
many self-intersection points (Theorem~\ref{theo:lower}).
For the class of tropical caterpillar lines 
we can give a tight upper bound (Theorem~\ref{theo:upper}).

\section{Preliminaries\label{se:prelim}}

\subsection{Tropical geometry\label{se:tropical}}

We review some concepts from tropical geometry. As general references
see \cite{gathmann-2006,itenberg-mikhalkin-shustin,rgst}.
Let $K$ be a field with a real valuation as introduced in 
Section~\ref{se:introduction} and let $\T(I)$ be the tropical variety of 
an ideal $I \lhd K[x_1, \ldots, x_n]$.
If $I$ is a prime ideal, then by the Bieri-Groves Theorem
$\T(I)$ is a pure $m$-dimensional polyhedral complex where 
$m=\dim(I)$ is the Krull dimension of the ideal \cite{bg}.

Define the local cone of a point $x$ of a polyhedral complex 
$\Delta\subseteq \RR^n$ as the set
\[LC_x(\Delta) \ := \ \{x+y\in \RR^n\ :\ \exists\ \varepsilon>0 \mbox{ such that } \{x+\rho y\ :\ 0\leq \rho\leq\varepsilon\}\subseteq\Delta\} \, . \]
The tropical variety
$\T(I)$ is \emph{totally concave}, which means that the convex hull of 
each local cone of a point $x$ is an affine subspace (see Figure \ref{convex})
\cite{bg}.

\begin{figure}[ht!]\setlength{\unitlength}{0.8cm}
\begin{center}
\begin{picture}(10,4)
\put(2,2){\color{blue} \vector(0,-1){1}}
\put(2,2){\color{blue}\vector(1,1){1}}
\put(2,2){\color{blue}\vector(-1,1){1}}
\put(1,3){\color{black}\line(-1,-1){1}}
\put(1,3){\color{black}\line(0,1){1}}
\put(3,3){\color{black}\line(1,0){1}}
\put(3,3){\color{black}\line(0,1){1}}
\put(2,1){\color{black}\line(1,0){1}}
\put(2,1){\color{black}\line(-1,-1){1}}
\put(2.4,1.8){\color{black}$x$}
\put(7,2){\color{black}\line(0,-1){2}}
\put(7,2){\color{black}\line(1,1){2}}
\put(7,2){\color{black}\line(-1,1){2}}
\put(7.2,1.8){\color{black}$LC_x(\T(I))$}
\end{picture}
\end{center}\caption{\label{convex}The local cone of a point}
\end{figure}
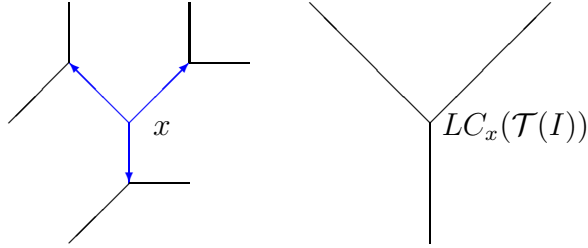

For tropical hypersurfaces $\T(f) := \T(\langle f \rangle)$ it is
well-known that $\T(f)$ is geometrically dual to a privileged 
subdivision of the Newton polytope $\new(f)$.
Namely we lift the points $\alpha$ in the support $\mathcal{A}$ of
$f$ into $\R^{n+1}$
using the valuations of the coefficients $c_{\alpha}$ as lifting values.
The set of those facets of 
$\hat{P}(f):= \text{conv}\{(\alpha,\val(c_{\alpha}))\, | \, \alpha \in \mathcal{A}\}$ which have an outward pointing normal with a negative last coordinate 
is called the \emph{lower hull}. The privileged subdivision is defined
by projecting down this lower hull back 
to $\R^n$ by forgetting the last coordinate. Denote by $C^{\vee}$
the dual cell of a cell $C$ in the tropical hypersurface.

With regard to the intersection of $k$ tropical hypersurfaces in $n$-space
($k \le n$), we use the following notation (see \cite{vigeland-2007}).
Let $f_1, \ldots, f_k \in K[x_1, \ldots, x_n]$,
and set $Y_i:= \mathcal{T}(f_i)$.
The intersection $Y := Y_1 \cap \cdots \cap Y_k$ 
is called \emph{proper} if $\dim Y = n-k$.
In order to study the intersection $Y$, it is useful to consider the
union $U := \bigcup_{i=1}^k Y_k$ as well, since $U$ is a tropical hypersurface,
$U = \T(f_1 \cdots f_k)$, and thus comes with a natural subdivision.

Let $C$ be a non-empty cell of a proper intersection $Y$. 
Then $C$ can be written as
$C = \bigcap_{i=1}^k C_i$, where $C_i$ is a cell of $Y_i$, minimal with $C\subseteq C_i$.
Consider $C$ as a cell of the union $U$. Then the dual cell $C^{\vee}$ of $C$
with regard to $U$ is given by the Minkowski sum
\[
  C^\vee \ = \ C_1^\vee + \dots + C_k^\vee \, .
\]
The intersection $Y$ is called \emph{transversal along $C$} if
\[
\dim(C^\vee) \ = \ \dim(C_1^\vee)+\dots+\dim(C_k^\vee) \, .
\] 
The intersection $Y$ is \emph{transversal} if for each
subset $J \subseteq \{1, \ldots, k\}$ of cardinality at least~2 the
intersection $\bigcap_{j \in J} Y_j$ is proper and transversal along each cell.

\begin{bsp}\label{ex:inter}Let $f_1=x+2y+z-4$, $f_2=3x-y+2z+1$ and the valuation $\val:\QQ\mapsto\RR_\infty$ be the $2$-adic valuation. Then $Y=\T(f_1)\cap\T(f_2)$ is a proper intersection, 
see Figure \ref{proper}.

\ifpictures
\begin{figure}[ht!]
\begin{center}
\hspace*{-5mm}
\begin{minipage}{3cm}\includegraphics[height=3cm]{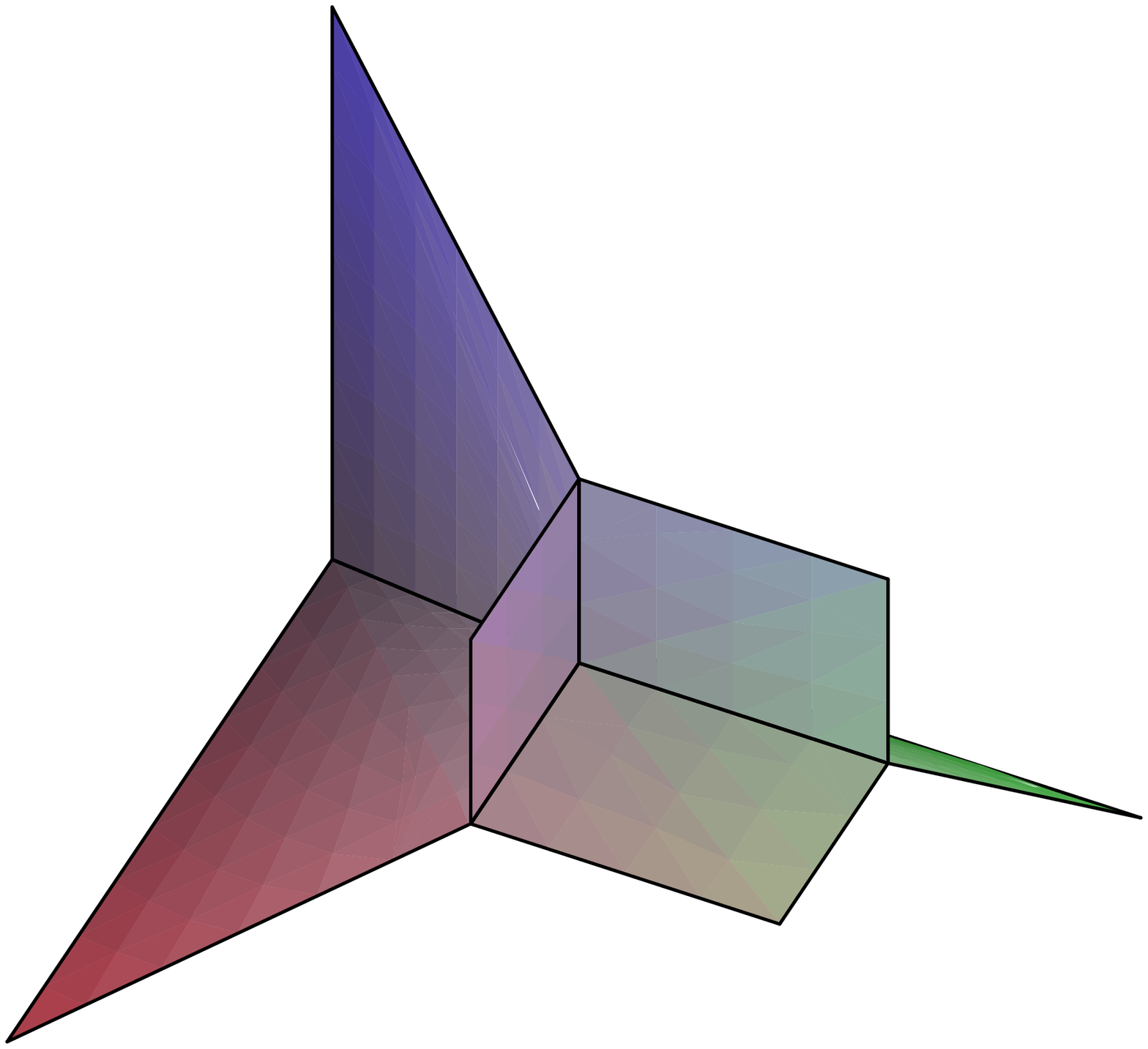}\end{minipage}\hspace*{0.5cm}
\begin{minipage}{1cm}$\bigcap$\end{minipage}\hspace{-0.5cm}
\begin{minipage}{3cm}\includegraphics[height=3cm]{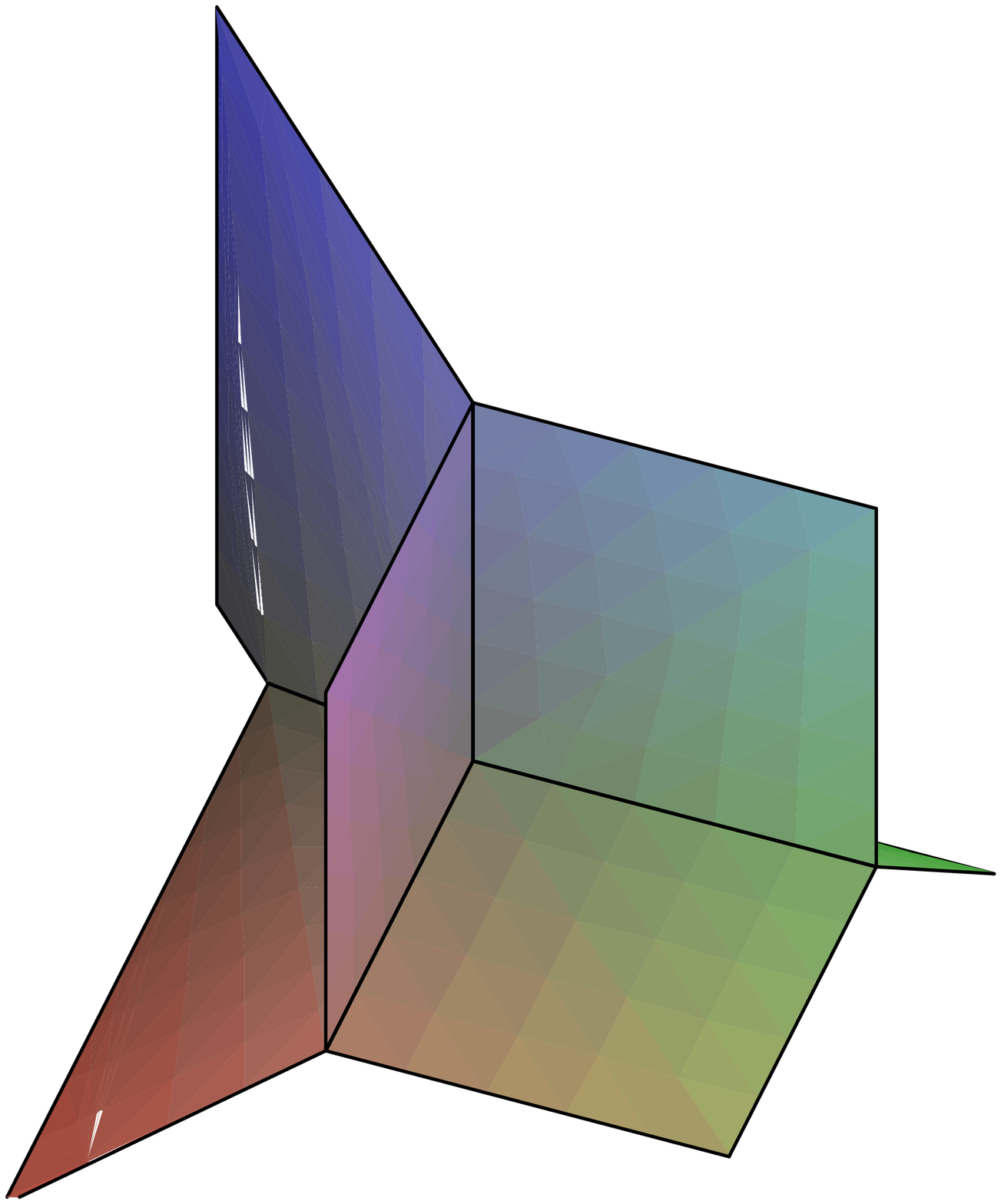}\end{minipage}\hspace*{0.7cm}
\begin{minipage}{1cm}$=$\end{minipage}\hspace*{0.0cm}
\begin{minipage}{3.5cm}\includegraphics[height=3.5cm]{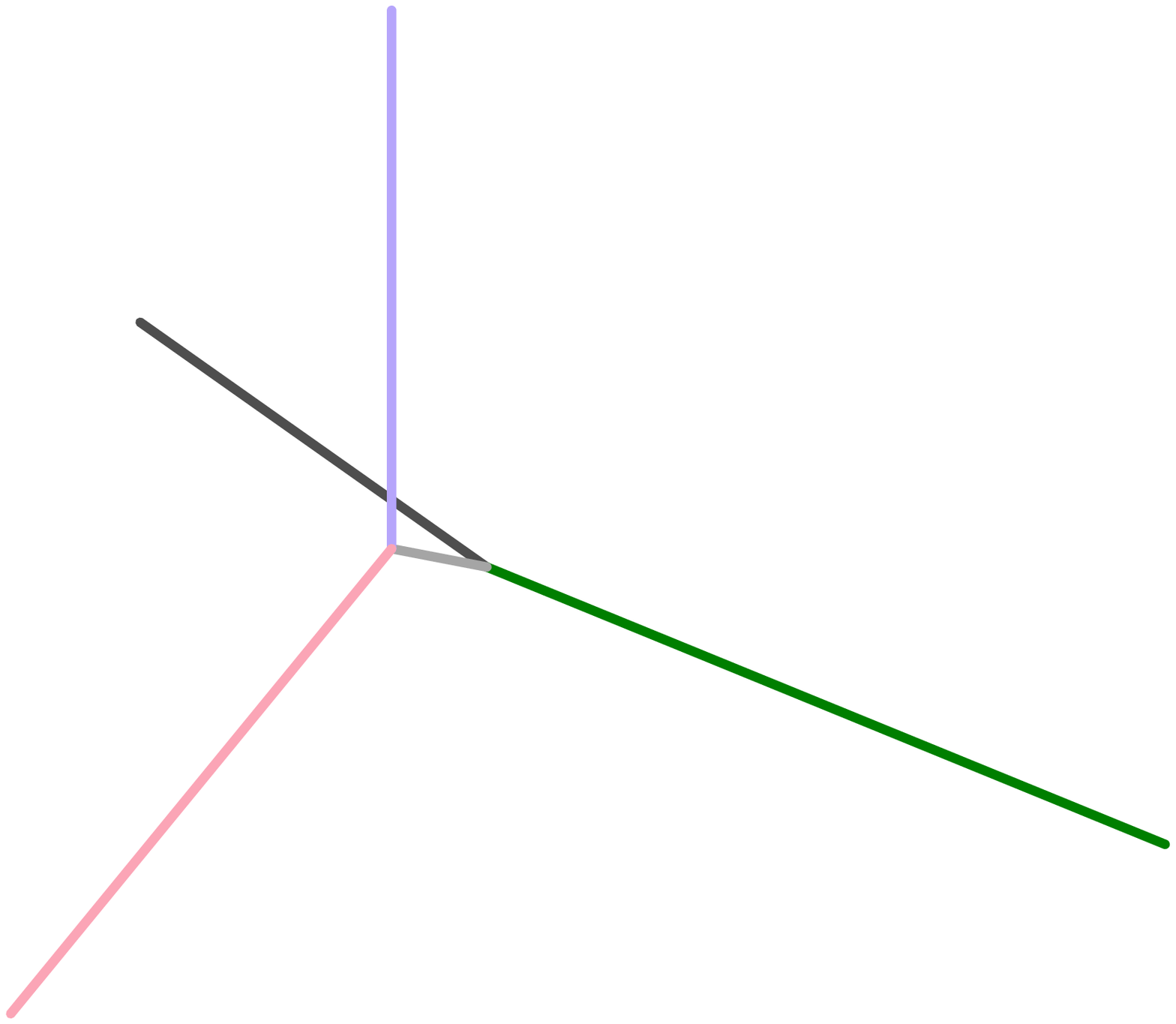}\end{minipage}
\vspace*{-0.8cm}

\begin{minipage}{4cm}
\includegraphics[height=4cm]{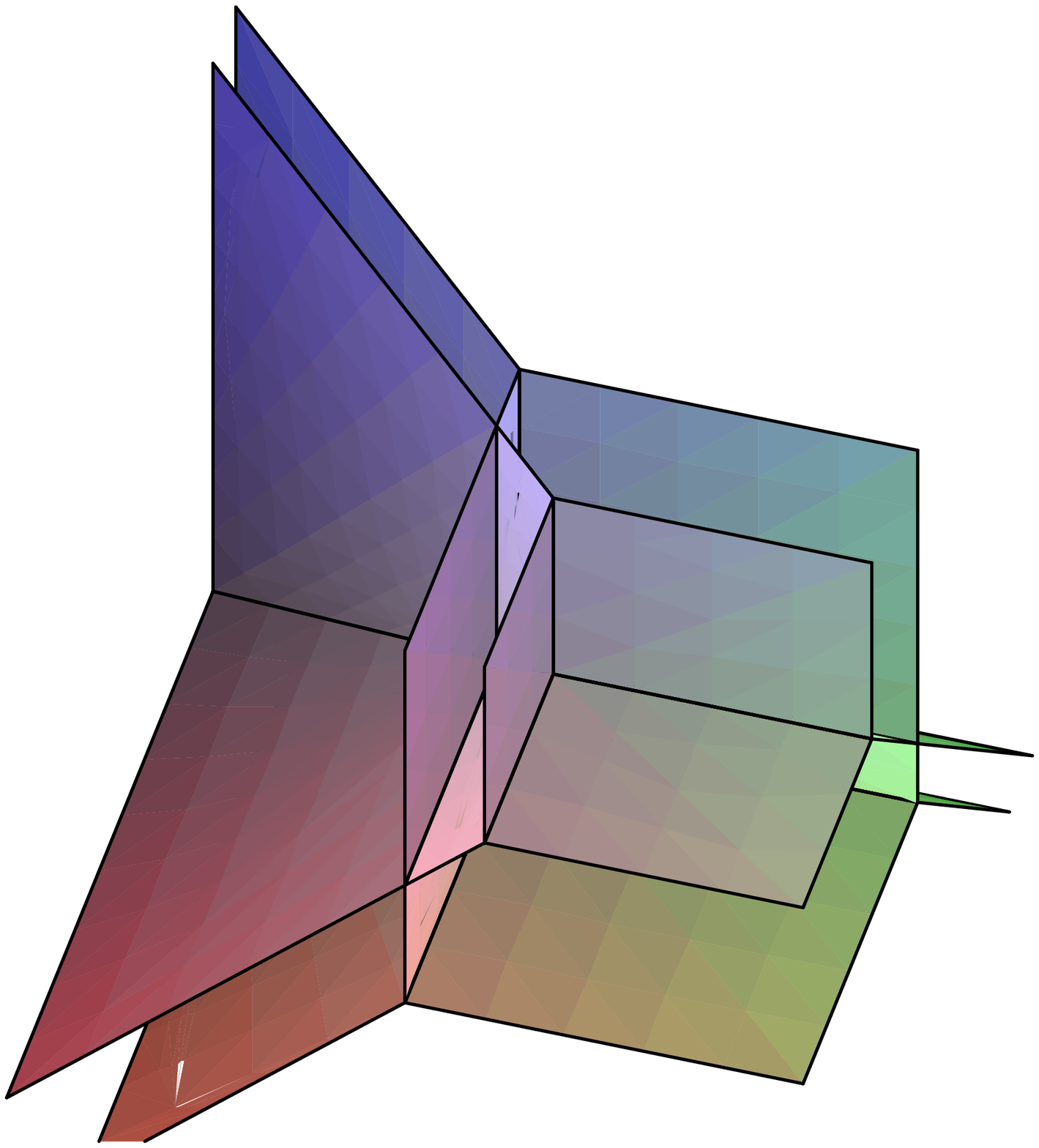}
\end{minipage}\hspace{2cm}
\begin{minipage}{5cm}
\vspace*{1cm}
\includegraphics[height=5cm]{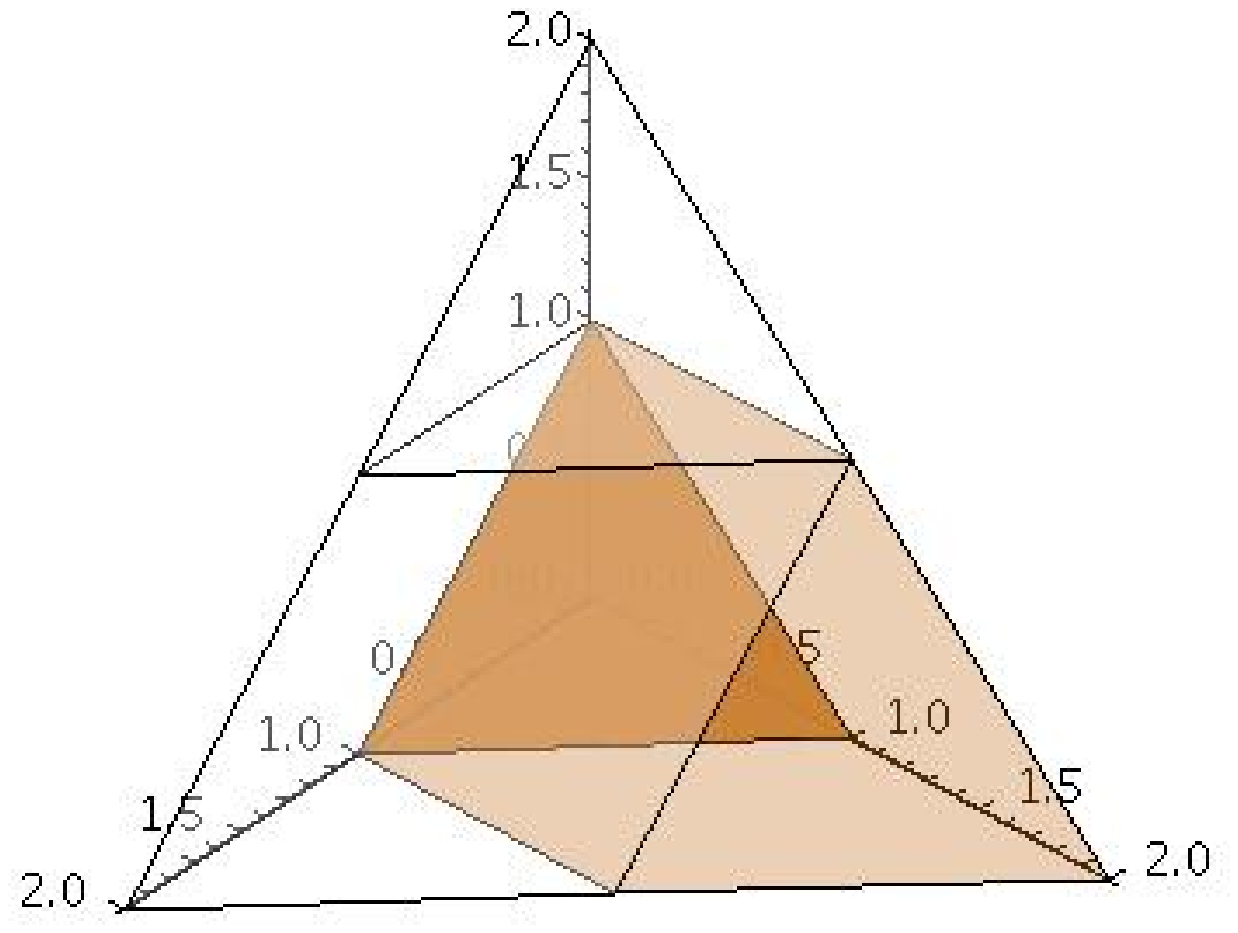}
\end{minipage}
\end{center}
\vspace*{-1.0cm}

 \caption{\label{proper}A proper intersection and the union
of two tropical hypersurfaces $\T(f_1)$ and $\T(f_2)$ in $\R^3$,
and the corresponding subdivision of the Newton polytope
of $\T(f_1 \cdot f_2)$.}
\end{figure}
\fi
\end{bsp}

A proper intersection $Y_1 \cap \cdots \cap Y_k$ is called a \emph{complete intersection} 
if
\[
  \mathcal{T}(\langle f_1, \ldots, f_k \rangle) \ = \
  \mathcal{T}(f_1) \cap \cdots \cap \mathcal{T}(f_k) \, .
\]

The polynomials $f_1, \ldots, f_k \in K[x_1, \ldots, x_n]\setminus\{0\}$
are called \emph{Newton-nondegenerate} 
\cite{bernstein-1975,esterov-khovanski-2006}
if for any collection of faces 
\[A_1\subseteq\new(f_1),\ldots, A_k\subseteq\new(f_k),\] 
such that the sum $A_1+\ldots+A_k$ is at most 
a $(k-1)$-dimensional face of the sum $\new(f_1)+\ldots+\new(f_k)$, 
the restrictions $f_1|_{A_1},\ldots,f_k|_{A_k}$ 
have no common zeros in $(\bar{K}^*)^n$. 
Here $f_i|_{A_i}$ is the sum of terms in $f_i$ with support in $A_i$.
Otherwise we call the polynomials $f_1, \ldots, f_k \in K[x_1, \ldots, x_n]$ 
\emph{Newton-degenerate.}

\subsection{Tropical bases\label{se:tropbases}}

For an ideal $I \lhd K[x_1, \ldots, x_n]$,
a finite generating set $\mathcal{F}$ of $I$ with
$\T(I) \ = \ \bigcap_{f\in \mathcal{F}}\T(f)$ is called a 
\emph{tropical basis} of $I$. Explicit bases are only known
for very special ideals (such as linear ideals or certain Grassmannians, 
see \cite{speyer-sturmfels-2004}).

A starting point for the present paper is the following
result by the authors \cite{hept-diss,hept-theobald-2009}
which guarantees tropical bases of small cardinality and
thus can be seen as a tropical analog to the Eisenbud-Evans-Theorem
from classical algebraic geometry (\cite{Eisenbud-Evans-1973}).
It was obtained by revisiting the regular projection technique of 
Bieri and Groves \cite{bg}.

\begin{prop}
\label{theo:tropbasis}
Let $I \lhd K[x_1,\ldots,x_n]$ be a prime ideal generated by 
the polynomials $f_1, \ldots, f_r$.
Then there exist $g_0,\ldots,g_n \in I$ with
\[  \T(I) \ = \ \bigcap_{i=0}^{n}\T(g_i) \, , \]
and thus $\mathcal{G} := \{f_1, \ldots, f_r, g_0, \ldots, g_n\}$
is a tropical basis for $I$ of cardinality $r+n+1$.
\end{prop}

The basis elements in 
Theorem~\ref{theo:tropbasis} may have large degrees.
In contrast to this, for linear ideals with constant coefficients
(i.e., $I \lhd \C[x_1, \ldots, x_n]$) the following lower bound 
states that if we want to require that all basis elements are
linear then small cardinality bases do not exist in general
\cite{bjsst,speyer-sturmfels-2004}. For $1\leq d\leq n$
there is a $d$-dimensional linear ideal $I$ in $\C[x_1,\ldots,x_n]$ 
such that any tropical basis of linear forms in $I$ has size
at least $\frac{1}{n-d+1}\binom{n}{d}$. See also
\cite{bjsst,speyer-sturmfels-2004} for Gr\"obner-related techniques
for constructing tropical bases.

\subsection{Mixed fiber polytopes}

Fiber polytopes have been introduced by Billera and Sturmfels 
in \cite{billera-sturmfels-92} (see also \cite{ziegler-95}) 
and generalize the concept of secondary 
polytopes \cite{gelfand-kapranov-zelevinsky-2008}. 
Let $\psi : \R^n \to \R^l$ be a linear map,
let $P$ be an $n$-polytope in $\R^n$,
and $Q$ be an $l$-polytope in $\R^l$ with $\psi(P) = Q$.
The fiber polytope $\Sigma_{\psi}(P)$ is defined as
\[
  \Sigma_{\psi}(P) \ = \ \int_Q (\psi^{-1}(x) \cap P) dx \ \subseteq \ \R^n \, ,
\]
where the integral on the right hand side is a 
Minkowski integral (see \cite{billera-sturmfels-92,ziegler-95}).

For $P_1, \ldots, P_r \subseteq \R^n$ and positive parameters
 $\lambda_1, \ldots, \lambda_r$ we consider the Minkowski sum
\[
  P_{\lambda} \ := \ \lambda_1 P_1 + \cdots + \lambda_r P_r \, .
\]
The fiber polytope $\Sigma_{\psi}(P_{\lambda})$ depends polynomially on the 
parameters $\lambda_1, \ldots, \lambda_r$, homogeneously of
degree $l+1$,
\begin{equation}
  \label{eq:mixed1}
  \Sigma_{\psi} (\lambda_1 P_1 + \cdots + \lambda_r P_r) \ = \
  \sum_{i_1 + \cdots + i_r = l+1} \lambda_1^{i_1} \lambda_2^{i_2} \cdots \lambda_r^{i_r} M_{i_1 \cdots i_r} 
\end{equation}
with some (uniquely determined) polytopes $M_{i_1 \cdots i_r}$.
For the case $r=l+1$, 
the \emph{mixed fiber polytope} $\Sigma_{\psi}(P_1, \ldots, P_r)$
is defined as the coefficient of the monomial $\lambda_1 \cdots \lambda_r$ 
in~\eqref{eq:mixed1},
\[
  \Sigma_{\psi}(P_1, \ldots, P_r) \ := \ M_{1 \cdots 1} \, .
\]
For $P := P_1 = \cdots = P_r$, we obtain (scaled)
ordinary fiber polytopes as special cases of mixed fiber polytopes,
$\Sigma_{\psi}(P, \ldots, P) \ = \ {r!}\ \Sigma_{\psi}(P)$.

Fiber polytopes can be expressed by Minkowski sums and 
formal differences of conventional fiber polytopes 
(see \cite{esterov-khovanski-2006}). For this, define 
a formal subtraction on the semigroup of 
polytopes with the Minkowski summation by 
\[P-Q=R \ :\Longleftrightarrow \ P=Q+R \, . \]
This gives the group of \emph{virtual polytopes}.
With this definition we can state:
\begin{theo} \label{th:alternatingsum}
For any polytopes $P_1, \ldots, P_r \subseteq \R^n$ we have
\begin{equation}
\label{eq:mixedfiber1}
\Sigma_\psi(P_1,\ldots,P_r) \ = \ \sum_{k=1}^n (-1)^{n+k} 
  \sum_{i_1 < \cdots < i_k}
  \Sigma_{\psi}(P_{i_1} + \cdots + P_{i_k}) \, .
\end{equation}
\end{theo}

\begin{proof} The proof is analogous to similar
statements on the mixed volume (see, e.g., \cite[Thm.\ IV.3.7]{ewald-1996}).
Denoting the right hand side of~\eqref{eq:mixedfiber1} by $g(P_1, \ldots, P_r)$,
we observe that for $\lambda_1, \ldots, \lambda_r > 0$ the expression
$g(\lambda_1 P_1, \ldots, \lambda_r P_r)$ is a polynomial in 
$\lambda_1, \ldots, \lambda_r$. 
For $P_1=\{0\}$, the definition of $g$ gives a telescoping sum,
\begin{eqnarray*}
 & & (-1)^{r+1}\cdot g(\{0\},P_2,\ldots,P_r) \\
  & = & \sum_{2\leq i\leq r}\Sigma_\psi(P_i) - 
  \left( \sum_{2\leq j\leq r}\Sigma_\psi(\{0\}+P_j)+\sum_{2\leq i< j\leq r}\Sigma_\psi(P_i+P_j) \right) \\
& & + \left( \sum_{2\leq j < k\leq r}\Sigma_\psi(\{0\}+P_j+P_k)  
  +\sum_{2\leq i< j < k \leq r}\Sigma_\psi(P_i+P_j+P_k) \right) \ \pm \ \ldots \\
& = & 0 \, , 
\end{eqnarray*}
where the last equality follows from re-arranging the parenthesis and 
$\Sigma_\psi(P)=\Sigma_\psi(\{0\}+P)$. 
As a consequence, the polynomial
$g(0\cdot P_1,\lambda_2\cdot P_2,\ldots,\lambda_r\cdot P_r)=0$ evaluates
to zero for all $\lambda_2,\ldots,\lambda_r$; i.e., it is the zero polynomial.
Thus, in the polynomial $g(\lambda_1\cdot P_1,\lambda_2\cdot P_2,\ldots,\lambda_r\cdot P_r)$, the coefficients of all monomials $\lambda_{i_1} \cdots \lambda_{i_r}$ 
with $1\notin \{i_1,\ldots,i_r\}$ vanish. By symmetry, this statement also holds
for all terms in which not all indices from $\{1, \ldots, r\}$
occur. Hence, there is only one monomial with nonzero coefficient, 
namely $\lambda_1 \cdots \lambda_r$. 
And this term must coincide with the mixed term in 
$\Sigma_{\psi}(\lambda_1 P_1 + \ldots + \lambda_r P_r)$.
So the corresponding coefficient has to be the mixed fiber polytope 
$\Sigma_{\psi}(P_1, \ldots, P_r)$.
\end{proof}

\section{Projections of tropical varieties via elimination theory\label{se:projelim}}

In this section, our main goal is to study some properties of
projections of tropical varieties from the viewpoint of elimination
theory. As already motivated in Section~\ref{se:tropbases},
we consider a tropical variety $\T(I)$ for some ideal $I \lhd K[x_1, \ldots, x_n]$ and a projection $\pi: \R^n \to \R^{m+1}$,
where $m$ is the dimension of $I$. We always assume that the projection
is \emph{rational}, i.e., that it is of the form $x \mapsto Ax$ with
some rational $(m+1) \times n$-matrix $A$ of rank $m+1$.
Then the (inverse) projection $\pi^{-1}(\pi(\T(I)))$ 
is a tropical variety. 

In terms of elimination theory,
the set $\pi^{-1} \pi(\T(I))$ can be characterized as follows.
Fix a basis $v^{(1)},\ldots, v^{(l)} \in \Z^n$ spanning the kernel 
of the projection $\pi$, where $l:=n-(m+1)$.
For $f\in I$ let
\[
  \tilde{f} \ = \ f(x_1 \prod_{j=1}^l \lambda_j^{v_1^{(j)}}, \ldots, x_n \prod_{j=1}^l \lambda_j^{ v_n^{(j)}}) \, .
\]
Then the ideal 
$J \lhd K[x_1, \ldots, x_n, \lambda_1^{\pm}, \ldots, \lambda_l^{\pm}]$
is defined by 
\[J \ := \ \langle \tilde{f}:~f\in I\rangle \ \lhd \ 
  K[x_1,\ldots,x_n,\lambda_1^{\pm},\ldots,\lambda_l^{\pm}] \, . 
\]
It can easily be checked that $J$ is generated by 
$\tilde{f}_1,\ldots,\tilde{f}_s$.
The following characterization was shown in \cite{hept-theobald-2009}.

\begin{prop}
\label{th:piinvpi}
Let $I \lhd K[x_1, \ldots, x_n]$ be an $m$-dimensional prime ideal and
$\pi : \RR^n \to \RR^{m+1}$ be a rational projection.
\begin{enumerate}
\item[(a)] Then $\pi^{-1} \pi(\T(I))$ is a tropical variety with
\begin{equation}
\label{eq:piinversepi}
  \pi^{-1} \pi(\T(I)) \ = \  \T(J \cap K[x_1, \ldots, x_n]) \, .
\end{equation}
\item[(b)]
  If $\pi(\T(I))$ is $m$-dimensional then 
  $\pi^{-1} \pi(\T(I))$ is a tropical hypersurface.
\end{enumerate}
\end{prop}

Geometrically, a polynomial $f$ and its corresponding polynomial 
$\tilde{f}$ are related by the following simple observation:

\begin{observation}
\label{ob:observation}
Let $f \in K[x_1, \ldots, x_n]$ and $z = (z_1, \ldots, z_n) \in V(f)$.
If $(w_{n+1}, \ldots, w_{n+l}) \in \R^l$ 
and $y$ denotes an element of valuation~1, then the point
\[
  \bar{z} \ := \ \left( \frac{z_1}{y^{\sum_{j=1}^l v_1^{(j)} w_{n+j}}}, \ldots,
                            \frac{z_n}{y^{\sum_{j=1}^l v_n^{(j)} w_{n+j}}},
                            y^{w_{n+1}}, \ldots, y^{w_{n+l}} \right) \in (K^*)^{n+l}
\]
is contained in $V(\tilde{f})$ with 
$\val(\overline{z}_i) = \val z_i - \sum_{j=1}^l v_i^{(j)} w_{n+j}$, $1 \le i \le n$.
In particular, the first $n$ components
of $\bar{z}$ are of the form
$(\frac{z_1}{y^{b_1}}, \ldots, \frac{z_n}{y^{b_n}})$ with 
$b \in \ker \pi$. 
\end{observation}

The next statement shows that many interesting properties carry over 
from $I$ to $J$.

\begin{lemma}\label{le:carryover}
Given polynomials $f_1, \ldots, f_{n-m} \in K[x_1, \ldots, x_n]$, let
$Y = \bigcap_{i=1}^{n-m} \T(f_i)$ and 
$\tilde{Y} = \bigcap_{i=1}^{n-m} \T(\tilde{f}_i)$.
\begin{enumerate}
\item[(a)] If the intersection is a proper intersection 
then the intersection $\tilde{Y}$ is a proper intersection.
\item[(b)] If the intersection is transversal then the intersection
  $\tilde{Y}$
is transversal.
\item[(c)] If $f_1, \ldots, f_{n-m}$ are New\-ton-nondegenerate
  then $\tilde{f}_1, \ldots, \tilde{f}_{n-m}$ are New\-ton-non\-de\-gen\-erate.
\item[(d)] If the intersection is complete then the intersection 
  $\tilde{Y}$ is complete.
\end{enumerate}
\end{lemma}

Let $\rho : \RR^{n+l} \to \RR^n$ be the projection forgetting the last $l$ coordinates,
and let $\rho^{\circ} : \RR^{n+l} \to \RR^l$ be the projection onto the last $l$ coordinates
(i.e., the map forgetting the first $n$ coordinates). 

\begin{proof}
The set of vectors $\alpha$ in the support of $f_i$ is in 1-1-correspondence
with the set of vectors in the support of $\tilde{f}_i$ by
the injective linear mapping
\[
  (\alpha_1, \ldots, \alpha_n) \ \mapsto \ (\alpha_1, \ldots, \alpha_n,
  \sum_{i=1}^n v^{(1)}_i \alpha_i, \ldots, \sum_{i=1}^n v^{(l)}_i \alpha_i)
\, .
\]
This mapping implies a canonical 1-1-correspondence between the dual
cells of $f_i$ and $\tilde{f}_i$. Statements (a) and (b) of the theorem
are an immediate consequence of this observation.

In order to prove (c)
let $\tilde{A}_i\subseteq\new(\tilde{f}_i)$, $i=1,\ldots,n-m$, be a collection
of faces such that the sum $\tilde{A}_1+\dots+\tilde{A}_{n-m}$ is at most
an $(n-m-1)$-dimensional face of
$\new(\tilde{f}_1)+\dots+\new(\tilde{f}_{n-m})$. Then any common zero
$(c_1,\ldots,c_{n+l})\in(\bar{K}^*)^{n+l}$ induces a common zero
\[
  \Big(c_1\prod_{j=1}^l c_{n+j}^{v_1^{(j)}},\ldots,c_n\prod_{j=1}^l
c_{n+j}^{v_n^{(j)}}\Big)\in(\bar{K}^*)^n
\] of
$\rho(\tilde{A}_1)=A_1,\ldots,\rho(\tilde{A}_{n-m})=A_{n-m}$ which are faces
of $\new(f_i)$ such that $\sum_{i=1}^{n-m} A_i$ is at most an
$(n-m-1)$-dimensional face of $\sum_{i=1}^{n-m}\new(f_i)$. This proves the
assertion.

Concerning (d), let $w \in \bigcap_{i=1}^{n-m} \T(\tilde{f}_i)$.
By definition of the $\tilde{f}_i$ there exists a point
$u \in \bigcap_{i=1}^{n-m} \mathcal{T}(f_i)$ with
$u_i = (w_i + \sum_{j=1}^l v_i^{(j)} w_{n+j})$, $1 \le i \le n$.
Since the intersection $Y$ is complete, 
$u \in \mathcal{T}(\langle f_1, \ldots, f_{n-m}\rangle)$.
Further, since $\bigcap_{i=1}^{n-m} \T(\tilde{f}_i)$ is closed,
we can assume without loss of generality that there exists
a point $z = (z_1, \ldots, z_n) \in V(f_1, \ldots, f_{n-m})$
with $\val(z_i) = w_i + \sum_{j=1}^l v_i^{(j)} w_{n+j}$.
If $y$ denotes an element of valuation~1, then the point $\bar{z}$
defined in Observation~\ref{ob:observation} satisfies
$\val(\bar{z}) = (w_1, \ldots, w_n,w_{n+1}, \ldots, w_{n+l})$
and $\bar{z} \in V(\tilde{f}_1, \ldots, \tilde{f}_{n-m})$.
Thus $w \in \T(\langle \tilde{f}_1, \ldots, \tilde{f}_{n-m} \rangle)$,
which completes the proof.
\end{proof}

\section{Combinatorics of projections of tropical varieties\label{se:combinatorics}}

\subsection{The dual subdivision}

Let $\pi:\RR^n\to\RR^{m+1}$ be again a projection represented by the matrix $A$ and $I\lhd\  K[x_1,\ldots,x_n]$ be an $m$-dimensional ideal. Assume that 
$\pi^{-1} \pi(\T(I))$ is a tropical hypersurface.
To describe the Newton polytope of a polynomial 
$f$ with $\pi^{-1} \pi(\T(I)) = \T(f)$
we consider a complementary linear map
\[\pi^\circ:\RR^n \ \to \ \RR^{n-m-1}=\RR^l\]
whose kernel is the rowspace of $A$. Then the following theorem was
shown by Esterov and Khovanski \cite{esterov-khovanski-2006}
and Sturmfels, Tevelev, and Yu 
(\cite{sturmfels-tevelev-2008,sturmfels-tevelev-yu-2007} and \cite[Thm. 4.1]{sturmfels-yu-2008}).

\begin{prop}\label{theo:styu}
Let $f_1, \ldots, f_{n-m} \in K[x_1, \ldots, x_n]$ be Newton-nondegenerate polynomials
and let the intersection $\bigcap_{i=1}^{n-m} \T(f_i)$ be complete.
Then the Newton polytope of $\pi^{-1} \pi(\T(\langle f_1,\dots,f_{n-m}\rangle))$ is affinely isomorphic to the mixed fiber polytope $\Sigma_{\pi^\circ}(\new(f_1),$ $\ldots,$ $\new(f_{n-m}))$.
\end{prop}

Next we provide an alternative characterization of the Newton polytope
based on the elimination characterization in Proposition~\ref{th:piinvpi}.
To apply this, we consider the coordinate projection
$\rho^{\circ}:\RR^{n+l}\to\RR^n$ which forgets the last $l$ coordinates
and $\rho:\RR^{n+l}\to\RR^l$ which forgets the first $n$ coordinates. 

\begin{theo} \label{th:mixedfiber1}
Let $f_1, \ldots, f_{n-m} \in K[x_1, \ldots, x_n]$ be Newton-nondegenerate, and let
the intersection 
\[Y \ = \ \bigcap_{i=1}^{n-m} Y_i \ = \ \bigcap_{i=1}^{n-m} \T(f_i) \ =\ \T(I)\]
be complete with $I := \langle f_1, \ldots, f_{n-m} \rangle$. 
Then the Newton polytope of the 
tropical hypersurface $\pi^{-1} \pi(\mathcal{T}(I))$ is affinely
isomorphic to
\begin{equation}
  \label{eq:alternfiberpoly}
  \rho(\Sigma_{\rho^{\circ}}(\new(\tilde{f}_1), \ldots, \new(\tilde{f}_{n-m}))) \, .
\end{equation}
Hence, up to an affine isomorphism, the mixed fiber polytopes
$\Sigma_{\pi^{\circ}}(\new(f_1), \ldots, \new(f_{n-m}))$ and
$\rho(\Sigma_{\rho^{\circ}}(\new(\tilde{f}_1), \ldots, \new(\tilde{f}_{n-m})))$
coincide.
\end{theo}

\begin{proof} Starting from 
the elimination characterization in Proposition~\ref{th:piinvpi}, we
know that $\pi^{-1} \pi(\mathcal{T}(I)) \ = \ \mathcal{T}(J \cap K[x_1, \ldots, x_n])$,
where $J=\langle\tilde{f}_1,\ldots,\tilde{f}_{n-m}\rangle$.
By Lemma~\ref{le:carryover} the polynomials 
$\tilde{f}_1,\ldots,\tilde{f}_{n-m}$ are Newton-nondegenerate
and the intersection $\bigcap_{i=1}^{n-m}\T(\tilde{f_i})$ is complete.
Then Theorem \ref{theo:styu} implies that
the Newton polytope of the defining polynomial of the right
hand side is (up to an affine isomorphism) given by the mixed 
fiber polytope
\[
  \Sigma_{\rho^{\circ}}(\new(\tilde{f}_1), \ldots, \new(\tilde{f}_{n-m})) \ 
  \subseteq \ \R^{n+l} \, .
\]
Applying the canonical projection $\rho$ which maps the mixed fiber polytope 
isomorphic onto its image proves the representation~\eqref{eq:alternfiberpoly}.
The coincidence with
$\Sigma_{\pi^{\circ}}(\new(f_1), \ldots, \new(f_{n-m}))$
then follows from  Theorem~\ref{theo:styu}.
\end{proof}

In the following, 
we study the subdivision of the Newton polytope of the defining polynomial of $\pi^{-1}\pi(\T(I))$. Each cell of that subdivision provides a local description of $\pi^{-1}\pi(\T(I))$.
These cells are described by mixed fiber polytopes. We will also show how to 
patchwork these local fiber polytopes.

\begin{remark} 
\label{re:relative}
We remark that rather considering the dual subdivision of
$\pi^{-1} \pi(\T(I))$ we could equivalently consider a dual subdivision
of $\pi(\T(I))$ (which can be considered as a tropical hypersurface relative 
to an $(m+1)$-dimensional hyperplane).
\end{remark}

We concentrate on the case of a transversal intersection, and we 
will always assume that the Newton nondegeneracy condition also holds for the 
local cells. 
In the local version of Corollary~\ref{th:mixedfiber1} we have to assume
that the preimage of a cell $\pi(C)$ is unique in $Y$.
Set $k:=n-m$.

\begin{lemma}\label{le:dualcell}
In the setup of Theorem~\ref{th:mixedfiber1},
let $C$ be a cell of $Y$ and $\pi^{-1}\pi(C)\cap Y=\{C\}$.
If $C^{\vee}$ denotes the dual cell of $C$ in 
the dual subdivision of $U = \bigcup_{i=1}^k \T(f_i)$
then $C^{\vee} \ = \ C_1^{\vee} + \cdots + C_k^{\vee}$
and the corresponding dual cell of $\pi^{-1}\pi(C)\subseteq\pi^{-1}\pi(\mathcal{T}(I))$ in the subdivision of the Newton polytope of the defining polynomial of $\pi(\T(I))$ is affinely isomorphic to
\[
\Sigma_{\pi^\circ} (C_1^{\vee}, \ldots , C_k^{\vee}) \, .
\] 
\end{lemma}
\begin{proof}
For $1 \le i \le k$ let $g_i$ be the polynomial with support given by $C_i$
whose coefficients are induced by $f_i$. 
Then the local cone of $\T(I)$ at $p \in C$ is given by $\LC_p(\T(I))=\T(\langle g_1, \ldots, g_k\rangle)$. If $C$ is the only preimage of $\pi(C)$, then 
\[\LC_{\pi(p)}\pi(\T(I)) \ = \ \pi(\LC_p\T(I)) \, .\]
By Proposition~\ref{theo:styu} the image $\pi(\T(\langle g_1, \ldots, g_k \rangle))$ is dual to 
$\Sigma_{\pi^\circ} (C_1^{\vee}, \ldots , C_k^{\vee})$.
\end{proof}

Note that each $C_i^\vee$ is the dual of a cell of dimension 
at most $n-1$, so it has dimension at least $1$. Thus 
the sum $C_1^\vee+\cdots+C_{n-m}^\vee$ is a 
mixed cell (cf.\ \cite{vigeland-2007}).

By Remark~\ref{re:relative}
we can consider $\pi(\T(I))$ as an $m$-dimensional complex in an $(m+1)$-dimensional space.
Every $j$-dimensional face $F$ of $\pi(\T(I))$ is either the projection of a unique
$j$-dimensional face of $\T(I)$ (see Lemma~\ref{le:dualcell}), or the intersection of the images of faces of $\T(I)$. Since every cell in the tropical hypersurface 
$\pi(\mathcal{T}(I))$ respectively $\pi^{-1} \pi(\mathcal{T}(I))$ 
arises in this way, we obtain:

\begin{theo}\label{ko:subdivision}
Let $k = n-m$, $I = \langle f_1, \ldots, f_k \rangle
\ \lhd \ K[x_1,\ldots,x_n]$ be an $m$-dimensional ideal, 
$\bigcap_{i=1}^{k} \mathcal{T}(f_i)$ be a complete intersection and
$\pi:\RR^n\to\RR^{m+1}$ be a rational projection.
If the condition of Newton-nondegeneracy is satisfied with respect to all
local cells then up to affine isomorphisms, the cells of the dual
subdivision of $\pi^{-1} \pi \T(I)$ are of the form
\begin{equation}
  \label{eq:p}
  \sum_{i=1}^p \Sigma_{\pi^\circ} (C_{i1}^{\vee}, \ldots, C_{i{k}}^{\vee})\mbox{ for some }p\in\NN.
\end{equation}
Here, $F_1, \ldots, F_p$ are faces of $\bigcap_{i=1}^k \T(f_i)$ 
and the dual cell of $F_i\subseteq U = \bigcup_{i=1}^k \T(f_i)$
is given by $F_i^\vee=C_{i1}^{\vee}+ \cdots+ C_{ik}^{\vee}$ with 
faces $C_{i1}, \ldots, C_{i k}$ of 
$\T(f_1), \ldots, \T(f_{k})$.
\end{theo}

Specifically, for $p=1$ the full-dimensional cells are of the form
$\Sigma_{\pi^\circ} (C_1^\vee, \ldots, C_{k}^\vee)$,
where $C_1^\vee+\dots + C_{k}^{\vee}$ is 
a mixed cell in $\new(f_1)+ \cdots + \new(f_{k})$.

\begin{proof}
For any cell $D\in\pi(\T(I))$ let $F_1,\ldots,F_p$ be the cells in $\T(I)$ minimal with $D\subseteq\pi(F_i)$ and $C_{i1}^{\vee}+ \ldots+ C_{ik}^{\vee}$ be the dual cells in the subdivision of $\new(f_1\cdots f_k)$. Then
\[ \LC_d \pi(\T(I)) \ = \ \bigcup_{i=1}^p\pi(\LC_{x^{(i)}}\T(I)) \, , \]
where $d$ is a point in $D$ and $x^{(i)}$ is the preimage of $d$ in $F_i$. As above the image $\pi(\LC_{x^{(i)}}\T(I))$ is dual to $\Sigma_{\pi^\circ} (C_{i1}^{\vee}, \ldots, C_{i{k}}^{\vee})$. Since the normal fan of a sum of two polytopes is the common refinement of the normal fans of the two polytopes, 
\[
  \bigcup_{i=1}^p\pi(\LC_{x^{(i)}}\T(I)) \ \mbox{is dual to}\ \sum_{i=1}^p \Sigma_{\pi^\circ} (C_{i1}^{\vee}, \ldots, C_{i{k}}^{\vee}) \, .
\]
This proves the claim.  
\end{proof}

Hence, every dual cell of the tropical hypersurface $\pi^{-1} \pi(\T(I))$
is indexed by some $p$-tuple of ``formal'' 
mixed fiber polytopes $\Sigma_{\pi^\circ} (C_{i1}^{\vee}, \ldots, C_{i{k}}^{\vee})$.

\subsection{Self-intersections\label{se:selfinter}}
If the projections $\pi(F_1),\pi(F_2)$ of two non-adjacent 
faces $F_1$ and $F_2$ of $\mathcal{T}(I)$ intersect in a point $a$, 
then $a$ is called a \emph{self-intersection point} of 
$\mathcal{T}(I)$ under~$\pi$. See Figure \ref{fig:sip} for an example.
From the dual viewpoint, a face of $\pi(\mathcal{T}(I))$ contains a
self-intersection point if the number $p$ of terms in the dual
characterization~\eqref{eq:p} is at least~2.

\begin{bsp1}
Figure~\ref{fig:sip} shows an example of a tropical line in $\R^3$
and a projection $\pi$ such that two non-adjacent half-rays of the
line intersect in the projection.
\end{bsp1}

\ifpictures
\begin{figure}[!ht]
\begin{center}
\resizebox{11cm}{!}{
 \input{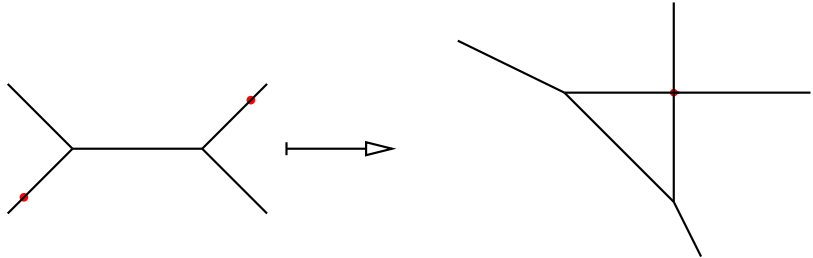}}
\end{center}
\vspace*{-1cm}

\caption{\label{fig:sip}A self-intersection point of 
 a line in $\R^3$ under a projection to $\R^2$}
\end{figure}
\fi

The number of self-intersections is not an invariant of the
tropical variety $\T(I)$, but depends on the choice of the
projection $\pi$.

Concerning upper bounds on the number of self-intersection points
for the case of curves, observe that every
self-intersection point is a singular point of the projection 
(which, as mentioned earlier, can be regarded as a tropical curve $\T(f)$ 
in the plane). Thus, the number of self-intersection
points is bounded by the number of singular points of $\T(f)$.
In Section~\ref{se:bounds}, we deal with more explicit bounds and
constructions for various special classes of tropical varieties

\subsection{Constructing the dual subdivision}

In all the preceding statements on the Newton polytopes and mixed fiber
polytopes, the Newton polytopes are only determined up to an affine isomorphism.
Since by Theorem~\ref{ko:subdivision}, the local cells in the subdivision
are fiber polytopes as well, the question arises how neighboring
cells fit together. In the following, we explain this patchworking of
the local mixed fiber polytopes. Theorem \ref{theo:conv} characterizes 
the offset of the fiber polytope of a facet of a simplex to the face of the fiber polytope of that simplex. Theorem \ref{theo:facemixed} is the generalization to faces of mixed fiber polytopes of mixed cells and Corollary \ref{cor:vector} gives us the desired patchworking of the cells in the dual subdivision of $\pi(\T(I))$.

For simplicity, we assume that we know a vertex $v$ of $\pi(\T(I))$
and the corresponding $m$-dimensional cell $C$
of the dual subdivision of $S := \pi(\T(I))$.
We explain how to pass over to a neighboring cell.
Locally around $v$, the tropical variety
$\T(I)$ defines the $m$-dimensional fan $\Gamma:=\LC_v(\T(I))$.
In order to determine the neighboring cell of $C$, 
we consider the 1-skeleton of $\Gamma$.
$\Gamma$ is geometrically dual to $\pi(C)$.

Consider a fixed direction vector $w$ of one of the rays 
of $\Gamma$. Let $v'$ be the neighboring vertex of $v$
on $S$ with regard to this ray, and let $\Gamma'$ be the
corresponding local fan. Further let $D$ be the dual
cell corresponding to $v'$.
Recall that $\pi^{\circ}$ has been fixed, and set $k:=n-m$.
Up to affine isomorphisms, Theorem~\ref{ko:subdivision} expresses 
$C$ and $D$ as
\begin{eqnarray*}
  C & = & \sum_{i=1}^{p_1} \Sigma_{\pi^\circ}(C_{i1}^{\vee}, \ldots, C_{ik}^{\vee}) \, , \\
  D & = & \sum_{i=1}^{p_2} \Sigma_{\pi^\circ}(D_{i1}^{\vee}, \ldots, D_{ik}^{\vee}) \, .
\end{eqnarray*}
Due to the affine isomorphisms,
these polytopes $C$ and $D$ do not necessarily share a common facet.
In order to characterize the translation involved it
suffices to characterize the offset from $C$ to the
``common face'' (up to a translation) of $C$ and $D$.
Denoting by $\arg \max$ the set of arguments at which a maximum is
attained, 
this common face is given by $\face_{w}(C) := \arg \max_{x \in C} w^T x$
and by $\face_{-w} (D)$ (up to translation) for some $w \in \R^n$.
We denote by $F$ this
face and use the notation
\[
  F \ = \ \sum_{i=1}^{q} \Sigma_{\pi^\circ}(F_{i1}^{\vee}, \ldots, F_{ik}^{\vee})
\]
with $q \le \min\{p_1,p_2\}$.
In the simplest case, $C$ consists of only one summand. Then in
the representation of $F$ one of the terms
$F_{ij}$ is a face of the corresponding $C_{ij}$ and the other $F_{ij}$
coincide with the corresponding $C_{ij}$.

\subsection*{Mappings $\pi^\circ$ to $\R^1$}
For the case that $\pi^\circ$ maps to $\R^1$, we will give explicit descriptions of the offsets of the mixed fiber polytopes of the mixed cells.
For a lattice polytope $P$,
let
\begin{equation}
\label{eq:pip}
{\pi^\circ}_{P} \ = \ \min_{x \in P} {\pi^\circ}(x) \; \text{ and } \;
{\pi^\circ}^{P} \ = \ \max_{x \in P} {\pi^\circ}(x) \, .
\end{equation}
Then
\[
  \Sigma_{\pi^\circ}(P) \ = \ 
  \sum_{i= {\pi^\circ}_P}^{{\pi^\circ}^P-1} \int_i^{i+1} ({\pi^\circ}^{-1}(x) \cap P) dx
  \ = \ \sum_{i={\pi^\circ}_P}^{{\pi^\circ}^P-1} ({\pi^\circ}^{-1}(i+\frac{1}{2}) \cap P) \, .
\]
Note that in general for a face $F$ of $P$ we do not have
that set-theoretically $\Sigma_{\pi^\circ}(F)$ is a face of $\Sigma_{\pi^\circ}(P)$.
As a consequence, in general for two polytopes $P_1$ and $P_2$ with
a common face the polytopes $\Sigma_{\pi^\circ}(P_1)$ and $\Sigma_{\pi^\circ}(P_2)$ 
do not have a common face. 

\begin{bsp} \label{ex:patchworkingex}
Let $\pi^\circ:\R^3\to \R, x\mapsto (1,1,1)\cdot x$, $P$ be the standard cube and $F$ the face 
\[\face_{(0,-1,0)}P \ = \ \conv\{(0,0,0),(1,0,1),(1,0,0),(0,0,1)\}.
\]
So $\pi^\circ(P)=[0,3]$, $\pi^\circ(F)=[0,2]$.
Then $\Sigma_\pi^\circ(P) = \sum_{i=0}^2 ({\pi^\circ}^{-1}(i + \frac{1}{2})\cap P)$
is a hexagon, and 
$\Sigma_{\pi^\circ}(F) = \sum_{i=0}^1 ({\pi^\circ}^{-1}(i + \frac{1}{2})\cap F)$
is a segment; see Figure \ref{fig5}. In particular, 
$\Sigma_{\pi^\circ}(F)+(1,\frac{1}{2},1)$ is a face of $\Sigma_{\pi^\circ}(P)$.
\end{bsp}

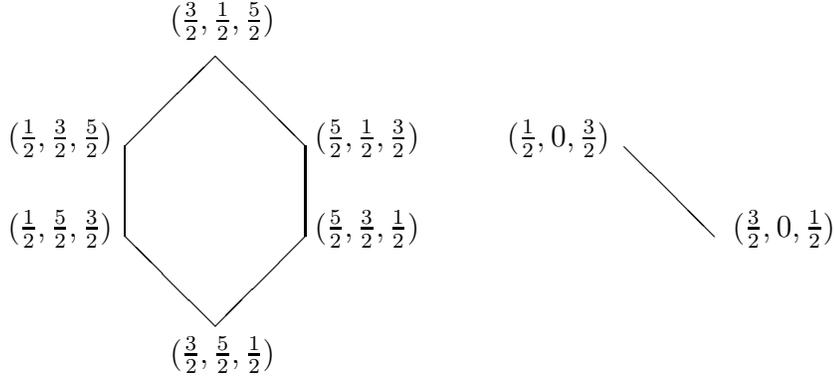
\begin{figure}[ht!]
\begin{center}
\begin{minipage}{5cm}
\setlength{\unitlength}{1.2cm}
\begin{picture}(3,4.5)
\put(2,1){\color{black}\line(-1,1){1}}
\put(1,2){\color{black}\line(0,1){1}}
\put(1,3){\color{black}\line(1,1){1}}
\put(2,4){\color{black}\line(1,-1){1}}
\put(3,3){\color{black}\line(0,-1){1}}
\put(3,2){\color{black}\line(-1,-1){1}}
\put(-0.3,2){\color{black}$(\frac{1}{2},\frac{5}{2},\frac{3}{2})$}
\put(3.1,2){\color{black}$(\frac{5}{2},\frac{3}{2},\frac{1}{2})$}
\put(3.1,3){\color{black}$(\frac{5}{2},\frac{1}{2},\frac{3}{2})$}
\put(-0.3,3){\color{black}$(\frac{1}{2},\frac{3}{2},\frac{5}{2})$}
\put(1.5,0.6){\color{black}$(\frac{3}{2},\frac{5}{2},\frac{1}{2})$}
\put(1.5,4.3){\color{black}$(\frac{3}{2},\frac{1}{2},\frac{5}{2})$}
\end{picture}
\end{minipage}\hspace{1.5cm}
\begin{minipage}{5cm}
\setlength{\unitlength}{1.2cm}
\begin{picture}(2,4.5)
\put(1,3){\color{black}\line(1,-1){1}}
\put(-0.3,3){\color{black}$(\frac{1}{2},0,\frac{3}{2})$}
\put(2.2,2){\color{black}$(\frac{3}{2},0,\frac{1}{2})$}
\end{picture}
\end{minipage}
\vspace*{-0.6cm}

\end{center}
\caption{\label{fig5}The fiber polytopes of $P$ and $F$} 
\end{figure}

First we characterize the affine isomorphism between the fiber polytope of a face of a polytope and the face of a fiber polytope in a simple case. For a 
lattice polytope $P$ and $i\in\NN$ define 
\[ [P]_i \ :=\ \arg \max_{x \in P \cap {\pi^\circ}^{-1}(i+\frac{1}{2})} w^Tx \, . 
\]

\begin{theo}\label{theo:conv}
Let $F$ be an $(n-1)$-dimensional lattice 
polytope in $\R^n$, ${\pi^\circ}:\R^n\to\R$, $v \in \Z^n \setminus \aff F$
and $P = \conv (F \cup \{v\})$. Let $w$ be an outer normal vector of the 
face $F$ of~$P$.
Then the face $\face_w(\Sigma_{\pi^\circ}(P))$ coincides with
\begin{equation}
  \label{eq:face1}
  \begin{cases}
    \Sigma_{\pi^\circ}(F) & \text{ if } {\pi^\circ}(v) \in {\pi^\circ}(F) \, , \\
  \Sigma_{\pi^\circ}(F) + \sum\limits_{i=\max_{x \in F} {\pi^\circ}(x)}^{{\pi^\circ}(v)-1} [P]_i & \text{ if } {\pi^\circ}(v) > \max\limits_{x \in F} {\pi^\circ}(x) \, , \\
  \Sigma_{\pi^\circ}(F) + \sum\limits^{\min_{x \in F} {\pi^\circ}(x)-1}_{{i=\pi^\circ}(v)} [P]_i & \text{ if } {\pi^\circ}(v) < \min\limits_{x \in F} {\pi^\circ}(x) \, ,
  \end{cases}
\end{equation}
where we assumed that all $\arg \max_{x \in P \cap {\pi^\circ}^{-1}(i+\frac{1}{2})} w^T x$ are unique.
\end{theo}

\begin{proof}
The points in~\eqref{eq:face1} are exactly the points in $\Sigma_{\pi^\circ}(P)$ which
maximize the objective function $x \mapsto w^T x$.
\end{proof}

Our characterization of the offset between two neighboring cells
in the dual subdivision is based upon the following theorem.
It describes the relation between the face of a 
mixed fiber polytope of $C$ and $D$ and the mixed fiber polytope of 
two faces of the two polytopes $C$ and $D$, where all faces 
maximize the same linear map. 

\begin{theo}\label{theo:facemixed}
Let ${\pi}:\R^n\to\R^{n-1}$ be a rational projection with
complementary map $\psi := \pi^{\circ}$, $C,D \subseteq \RR^n$ lattice
polytopes, and let $w\in\R^n$. Then
\begin{eqnarray*}
&& \Sigma_{\psi}(\face_w(C),\face_w(D))+ \sum_{i={\psi}_{C+D}}^{{\psi}_{\face_w(C+D)}-1} [C+D]_i
+ \sum_{i={\psi}^{\face_w(C+D)}}^{{\psi}^{C+D}-1}
[C+D]_i \\
& = & \face_w\Sigma_{\psi}(C,D) +
\sum_{i={\psi}_{C}}^{{\psi}_{\face_w(C)}-1} [C]_i +
\sum_{i={\psi}^{\face_w(C)}}^{{\psi}^C-1} [C]_i +
\sum_{i={\psi}_D}^{{\psi}_{\face_w(D)}-1}  [D]_i +
\sum_{i={\psi}^{\face_w(D)}}^{{\psi}^D-1} [D]_i \, ,
\end{eqnarray*}
where $\psi_P$ and $\psi^P$ are defined as in~\eqref{eq:pip}.
\end{theo}

\begin{proof}
Theorem \ref{th:alternatingsum} implies
$\Sigma_{\psi}(C+D) \ = \ \Sigma_{\psi}(C,D)+\Sigma_{\psi}(C)+\Sigma_{\psi}(D)$, which carries over to the faces of the polytopes,
\begin{equation}
\label{eq:sigma}
\face_w\Sigma_{\psi}(C+D) \ = \ \face_w\Sigma_{\psi}(C,D)+\face_w\Sigma_{\psi}(C)+\face_w\Sigma_{\psi}(D) \, .
\end{equation}
For a face $F=\face_w(P)$ of an $n$-polytope $P$ in $\R^n$ we have 
$\Sigma_{\psi}(F) = \face_w\int_{x\in \psi(F)} (P\cap{\psi}^{-1})dx$ and therefore
\[
  \label{eq:face3}
  \face_w(\Sigma_{\psi}(P))
  \ = \ \Sigma_{\psi}(F) + \face_w\int_{{\psi}(P)\setminus{\psi}(F)}\hspace*{-1cm}{\psi}^{-1}(x)\cap P\ dx \, .
\]
Applying this three times in~\eqref{eq:sigma} yields
\begin{eqnarray*}
& & \Sigma_{\psi}(\face_w(C+D)) + \face_w\int_{{\psi}(C+D)\setminus{\psi}(\face_w(C+D))}\hspace*{-3cm}{\psi}^{-1}(x)\cap (C+D)\  dx \\
& = & \face_w\Sigma_{\psi}(C,D)+\Sigma_{{\psi}}(\face_w(C)) + \face_w\int_{{\psi}(C)\setminus{\psi}(\face_w(C))}\hspace*{-2cm}{\psi}^{-1}(x)\cap C\  dx \\
& & + \ \Sigma_{{\psi}}(\face_w(D)) + \face_w\int_{{\psi}(D)\setminus{\psi}(\face_w(D))}\hspace*{-2cm}{\psi}^{-1}(x)\cap D\  dx \, .
\end{eqnarray*}
By linearity of the $\face_w$-operator and Theorem \ref{th:alternatingsum} we obtain 
\begin{eqnarray*}
& & \Sigma_{{\psi}}(\face_w(C),\face_w(D)) + \face_w\int_{{\psi}(C+D)\setminus{\psi}(\face_w(C+D))}\hspace*{-3cm}{\psi}^{-1}(x)\cap (C+D)\  dx \\
& = & \face_w\Sigma_{\psi}(C,D) + \face_w\int_{{\psi}(C)\setminus{\psi}(\face_w(C))}\hspace*{-2cm}{\psi}^{-1}(x)\cap C\  dx +\face_w\int_{{\psi}(D)\setminus{\psi}(\face_w(D))}\hspace*{-2cm}{\psi}^{-1}(x)\cap D\  dx \, .
\end{eqnarray*}
Then replacing the integrals by sums we get the assertion.
\end{proof}

\begin{kor}\label{cor:vector}
 In the case of a mixed cell $C+D$ where $D$ and $\pi^{\circ}(D)$
are one-dimensional and
$\face_w(C)+D$ is a facet of $C+D$, the difference
\[v(C,D,w) \ := \ \Sigma_{\psi}(\face_w(C),D)-\face_w(\Sigma_{\psi}(C,D))\]
is a $0$-dimensional polytope.
\end{kor}

\begin{proof} Since $\face_w(C) + D$ is a facet, we have $\face_w(D) = D$
and thus the last two terms in Theorem~\ref{theo:facemixed} vanish.
Further, since ${\psi} := \pi^{\circ}$ is a projection to $\RR$,
\begin{eqnarray*}
\min{\psi}(\face_w C+ D)-\min{\psi}(C+D) & = & \min{\psi}(\face_wC)-\min{\psi}(C) \\
\text{ and } \: 
  \max{\psi}(C+D)-\max{\psi}(\face_w C+ D) & = & \max{\psi}(C)-\max{\psi}(\face_wC) \, ,
\end{eqnarray*}
and thus there is a 1-1-correspondence
between the terms of the two sums on the left hand side 
and of the right hand side
of the equation in Theorem~\ref{theo:facemixed}.
Since the difference between the corresponding terms is just a
vector (i.e., a 0-dimensional polytope), the statement follows.
\end{proof}

If $C_1+D_1$ and $C_2+D_2$ are two neighboring cells with respect to a vector $w$ in the Minkowski sum of the two Newton polytopes, then the offset
between the cells
$\Sigma_{\pi^\circ}(C_1,D_1)$ and $\Sigma_{\pi^\circ}(C_2,D_2)$ is 
$v(C_1,D_1,w)-v(C_2,D_2,-w)$.


\begin{bsp} \label{ex:monomialmap}
We consider the 2-adic valuation $\val:\QQ\mapsto\RR_\infty$. 
Let $f_1=x+2y+z-4$, $f_2=3x-y+2z+1$, 
and let \[\pi \, : \, \RR^3 \ \to \ \RR^2,\ x \ \mapsto \ \left(\begin{array}{ccc}1 & 2 & 0\\ 0 & 1 & 1\end{array}\right)\cdot x\]
be a projection with kernel $\langle (2,-1,1)\rangle$. Then a defining polynomial of $\pi^{-1}\pi\T(\langle f_1,f_2\rangle)$ is 
\[g \ := \ -338x-18z^2+483xyz+25yz^3+343y^2x^2 \, . \]

\noindent
After applying the monomial map 
\[B:\QQ[x,y,z] \ \to \ \QQ[x,y],\] \[x \ \mapsto \ x \, ,\; y \ \mapsto \ x^2y,\ z \ \mapsto \ y\]
induced by the projection matrix, we get 
\[B(g) \ := \ -338x-18y^2+483x^3y^2+25x^2y^4+343x^6y^2 \, . \]
This is a polynomial generating the image, $\T(B(g))=\pi(\T(I))$. 
For the subdivided Newton polytope and the tropical variety see Figure \ref{fig1}.

\ifpictures
\begin{figure}[t]
\begin{center}
\begin{minipage}{4cm}
\psfrag{1}{$1$}
\psfrag{x}{$x$}
\psfrag{y}{$y$}
\includegraphics[height=4cm]{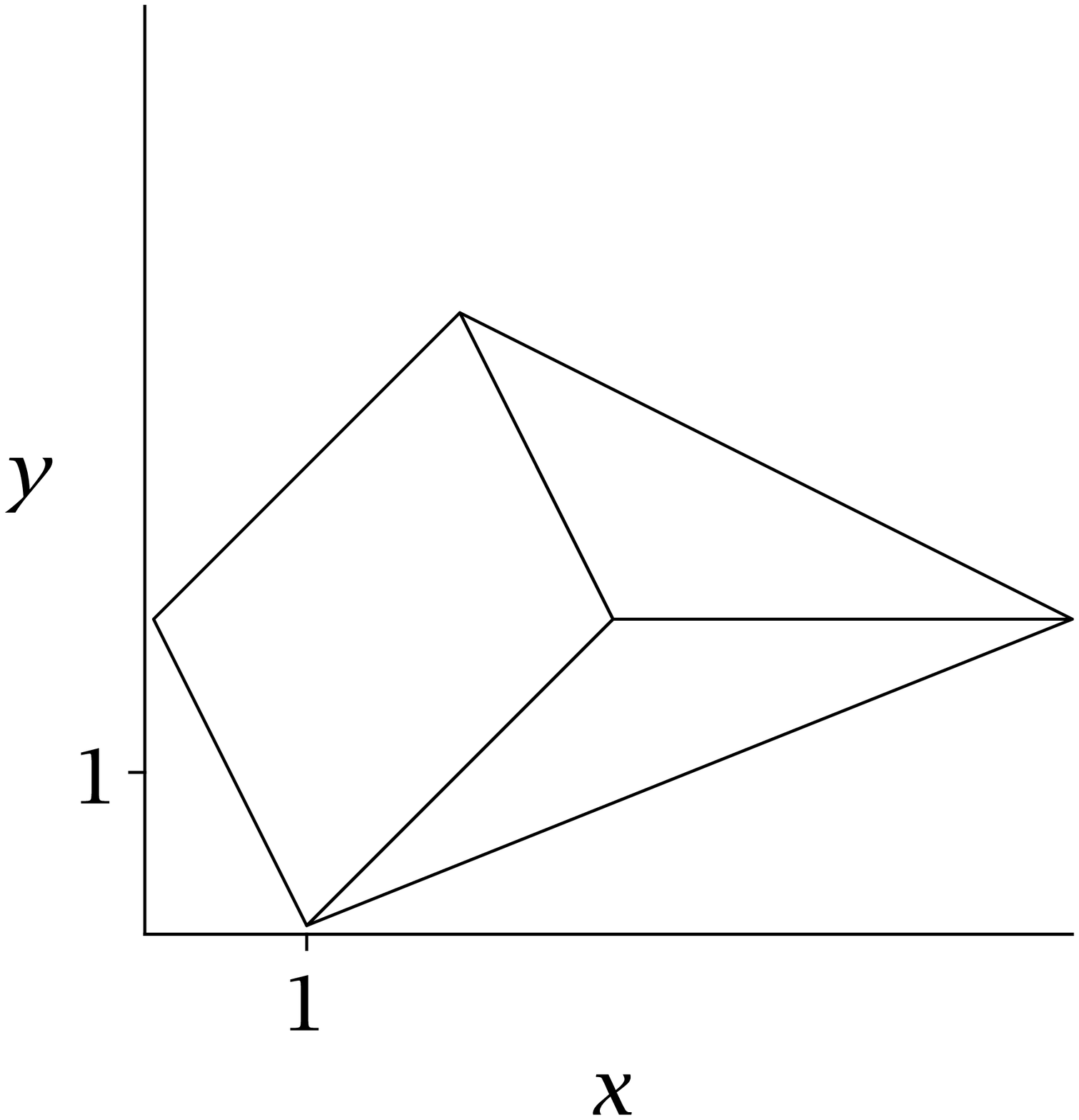}
\end{minipage} \hspace{1cm}
\begin{minipage}{5cm}
  \includegraphics[height=4cm]{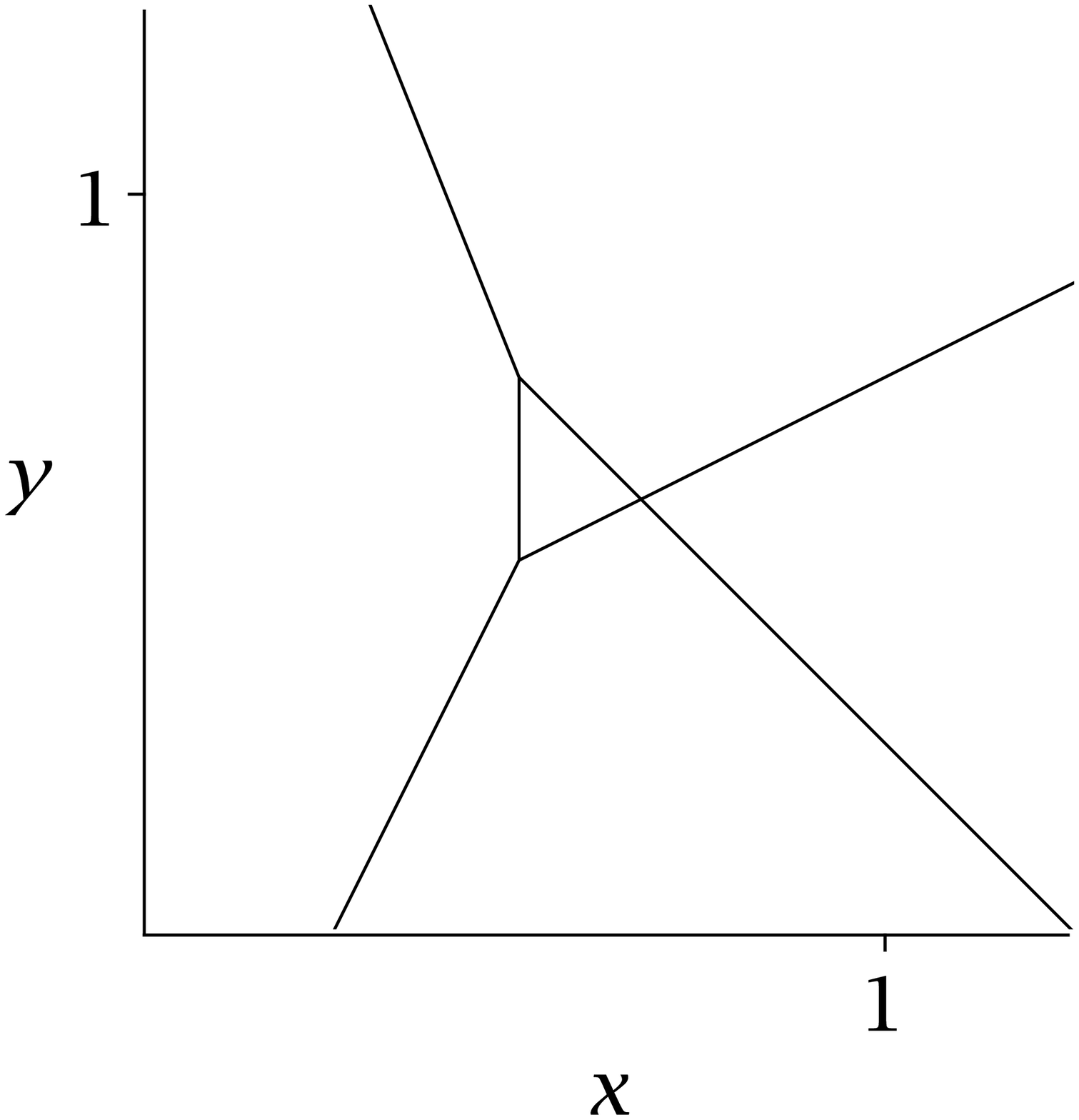}
\end{minipage}
\end{center}

\begin{center}
\begin{minipage}{4cm}
\setlength{\unitlength}{0.8cm}
\begin{picture}(4,5)
\put(0,1){\color{red}\line(0,1){3}}
\put(0,1){\color[rgb]{0,0.5,0}\line(1,0){3}}
\put(0,1){\color{red}\line(1,1){2}}
\put(3,1){\line(-1,1){3}}
\put(3,1){\color[rgb]{0,0.5,0}\line(-1,2){1}}
\put(0,4){\color{red}\line(2,-1){2}}
\put(-0.7,1){$w_0$}
\put(3.2,1){$w_1$}
\put(-0.7,4){$w_3$}
\put(2.3,3){$w_2$}
\put(1.5,1.5){\color[rgb]{0,0.5,0}$F_2$}
\put(0.5,2.5){\color{red}$F_1$}
\put(1.6,2.6){$F_3$}
\put(1.5,0.6){$F_4\uparrow$}
\end{picture}
\end{minipage} \hspace{1cm}
\begin{minipage}{4cm}
\setlength{\unitlength}{0.8cm}
\begin{picture}(4,5)
\put(0,1){\color{red}\line(0,1){3}}
\put(0,1){\color[rgb]{0,0.5,0}\line(1,0){3}}
\put(0,1){\color{red}\line(1,1){2}}
\put(3,1){\line(-1,1){3}}
\put(3,1){\color[rgb]{0,0.5,0}\line(-1,2){1}}
\put(0,4){\color{red}\line(2,-1){2}}
\put(-0.7,1){$v_0$}
\put(3.2,1){$v_1$}
\put(-0.7,4){$v_3$}
\put(2.3,3){$v_2$}
\put(1.5,1.5){\color[rgb]{0,0.5,0}$G_2$}
\put(0.5,2.5){\color{red}$G_1$}
\put(1.6,2.6){$G_3$}
\put(1.5,0.6){$G_4\uparrow$}
\end{picture}
\end{minipage} \qquad
\begin{minipage}{3.5cm}
 \resizebox{3.5cm}{!}{
 \input{Bilder/sumpolytopes.pstex_t}
}
\end{minipage}

\end{center}

\caption{\label{fig1}Top: $\new(B(g))$ and $\T(B(g))$.
Bottom: The Newton polytopes of $f_1$ and $f_2$ and their sum with the corresponding mixed cells.} 
\end{figure}
\fi

In the dual subdivision of $\new(f_1 \cdot f_2)$ there are two mixed $3$-cells which correspond to the two points of the tropical line $\T(\langle f_1,f_2\rangle)$. If the faces of the Newton polytopes of $f_1$ and $f_2$ are denoted as in Figure \ref{fig1} then the topdimensional mixed cells of the subdivision of $\new(f_1)+\new(f_2)$ are $F_3+[v_0,v_2]$ and $G_4+[w_1,w_3]$.

Applying Theorem~\ref{th:alternatingsum}, we compute the
corresponding mixed fiber polytopes $\Sigma_{\pi^\circ}(F_3,$ $[v_0,v_2])$ and $\Sigma_{\pi^\circ}(G_4,[w_1,w_3])$
where ${\pi^\circ}:\RR^3\to\RR,\ x\mapsto (2,-1,1)\cdot x$. Figure~\ref{fig3}
shows these fiber polytopes and their images after 
the translations stemming from Corollary~\ref{cor:vector};
these translations are $(-1,-1)$ and $(1,1)$, respectively.

\ifpictures
\begin{figure}[ht!]
\psfrag{1}{$1$}
\psfrag{x}{$x$}
\psfrag{y}{$y$}
\begin{center}
\begin{minipage}{4cm} 
 \includegraphics[height=4cm]{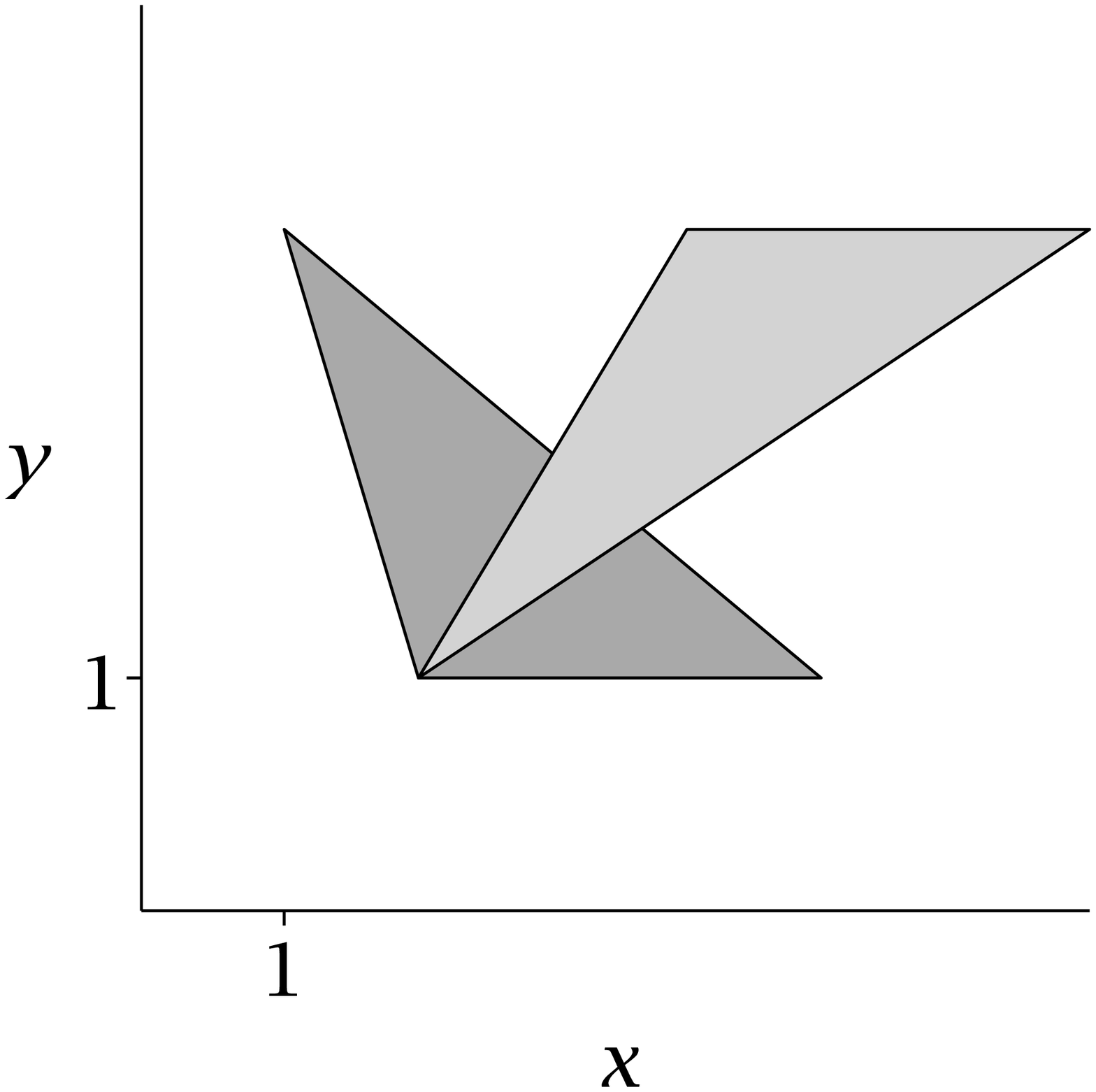}
\end{minipage}\hspace{2cm} 
\begin{minipage}{4cm}
 \includegraphics[height=4cm]{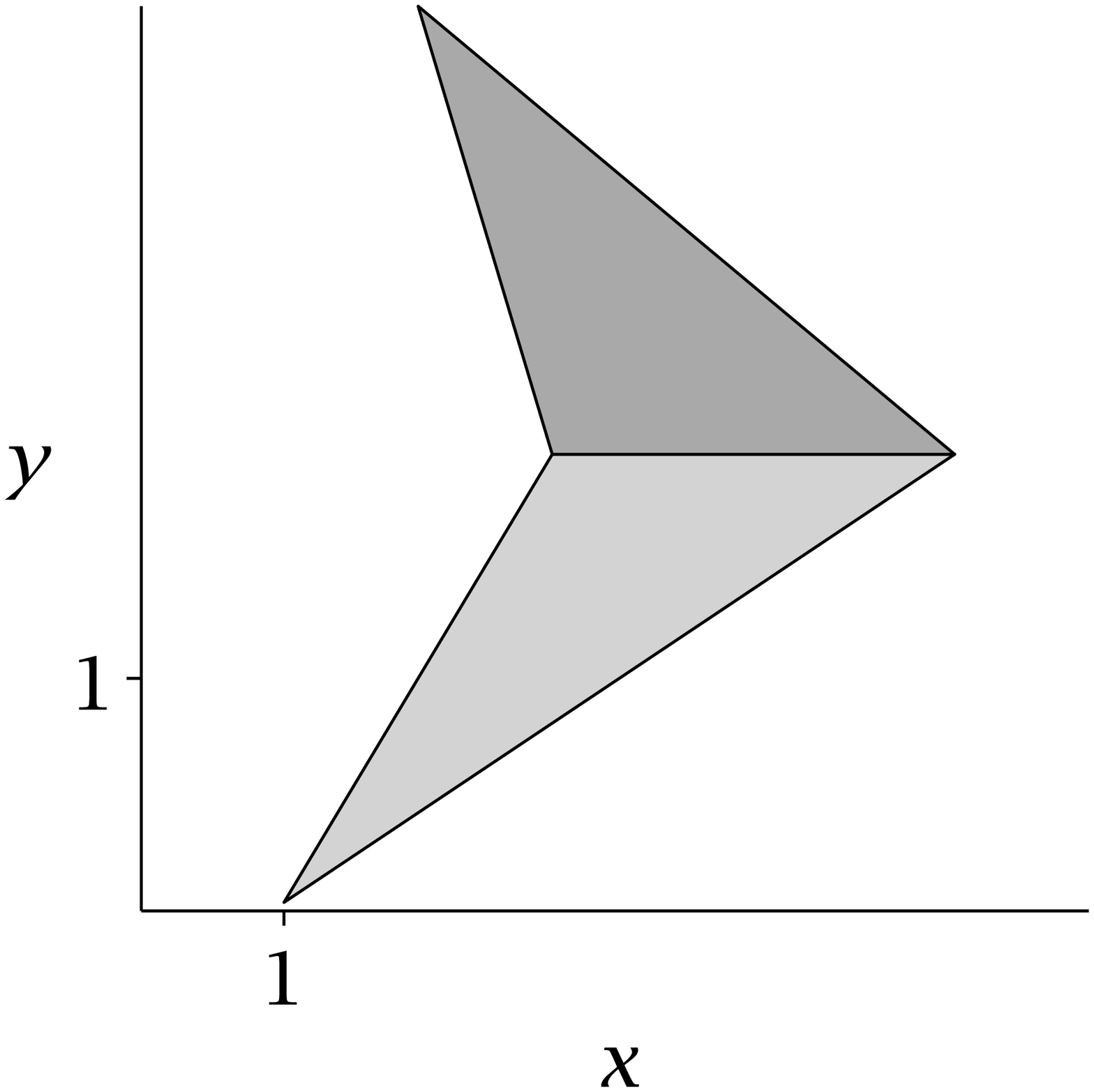}
\end{minipage}
\end{center}
\caption{\label{fig3} The fiber polytopes of the mixed cells} 
\end{figure}
\fi

The quadrangle in the subdivision in left upper picture of Figure~\ref{fig1}
is the dual cell of the self-intersection of the right picture.
The quadrangle is (up a translation) the sum of 
$\Sigma_\alpha([w_1,w_3],[v_0,v_1])$ and $
\Sigma_\alpha([w_2,w_3],[v_0,v_2])$. 
\end{bsp}

\section{Bounds on the number of self-intersections\label{se:bounds}}

In this section, we analyze the projections of tropical curves onto the plane
and derive some bounds on the complexity of the image.

Regarding the combinatorics of tropical curves,
in \cite{steffens-theobald-2009,vigeland-2007}
the number of vertices and the number of edges of a tropical transversal intersection 
curve was computed in dependence of the Newton polytopes of the
underlying tropical hypersurfaces.
Here, we give bounds on the number of vertices resp.\ self-intersections
(as defined in Section~\ref{se:selfinter}) of the image 
of a tropical curve. Most of our results refer to the case of lines.

Recall that a tropical line in $\RR^n$ has $n+1$ half-rays emanating into the directions $e^{(1)},\ldots,e^{(n)}$ and $-\sum_{i=1}^n e^{(i)}$. The combinatorial structure of (non-degenerate) lines has been studied in \cite{speyer-sturmfels-2004}. 
The combinatorial type of a non-degenerate line is a trivalent tree whose leaves are labeled by $1, \ldots, n+1$
(where the label $i$ for $1\leq i \le n$ represents the ray emanating in 
direction $e^{(i)}$ and $n+1$ the ray emanating in direction \mbox{$-\sum_{i=1}^n e^{(i)}$}).
The number of these trees is the Schr\"oder number
\[
  (2n-3)!! \ = \ 1 \cdot 3 \cdot 5 \cdots (2n-5) \cdot (2n-3) \, .
\]

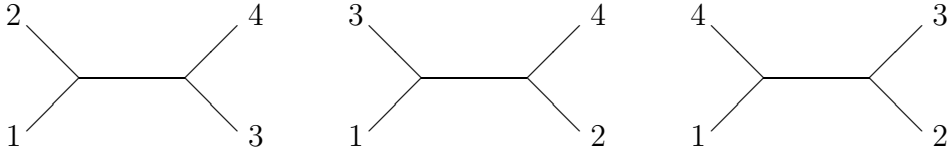
\begin{figure}[ht]\setlength{\unitlength}{0.7cm}
\begin{center}
\begin{minipage}{4cm}
\begin{picture}(4,2.5)
\put(1,1){\color{black}\line(1,0){2}}
\put(1,1){\color{black}\line(-1,1){1}}
\put(1,1){\color{black}\line(-1,-1){1}}
\put(3,1){\color{black}\line(1,1){1}}
\put(3,1){\color{black}\line(1,-1){1}}
\put(-0.4,-0.3){\color{black}$1$}
\put(-0.4,2){\color{black}$2$}
\put(4.2,-0.3){\color{black}$3$}
\put(4.2,2){\color{black}$4$}
\end{picture}
\end{minipage}
\quad
\begin{minipage}{4cm}
\begin{picture}(4,2.5)
\put(1,1){\color{black}\line(1,0){2}}
\put(1,1){\color{black}\line(-1,1){1}}
\put(1,1){\color{black}\line(-1,-1){1}}
\put(3,1){\color{black}\line(1,1){1}}
\put(3,1){\color{black}\line(1,-1){1}}
\put(-0.4,-0.3){\color{black}$1$}
\put(-0.4,2){\color{black}$3$}
\put(4.2,-0.3){\color{black}$2$}
\put(4.2,2){\color{black}$4$}
\end{picture}
\end{minipage}
\quad
\begin{minipage}{4cm}
\begin{picture}(4,2.5)
\put(1,1){\color{black}\line(1,0){2}}
\put(1,1){\color{black}\line(-1,1){1}}
\put(1,1){\color{black}\line(-1,-1){1}}
\put(3,1){\color{black}\line(1,1){1}}
\put(3,1){\color{black}\line(1,-1){1}}
\put(-0.4,-0.3){\color{black}$1$}
\put(-0.4,2){\color{black}$4$}
\put(4.2,-0.3){\color{black}$2$}
\put(4.2,2){\color{black}$3$}
\end{picture}
\end{minipage}
\vspace*{0.2cm}

\end{center}\caption{\label{fi:linetypes}
The three combinatorial types $[12,34]$
$[13,24]$, and $[14,23]$ of a tropical line in $\R^3$.
}
\end{figure}

\begin{bsp}
For a tropical line in $\R^3$ there are three combinatorial types,
as depicted in Figure~\ref{fi:linetypes} (see \cite{rgst}).
In $\RR^4$ there are 15 different non-degenerate types of lines. 
\end{bsp}

A tropical
line is called a \emph{caterpillar} if the graph of its combinatorial type 
has diameter $n$ (and thus is maximally possible). See Figure~\ref{fig:caterpillar}. For $n \in \{3,4\}$
all (non-degenerate) tropical lines are caterpillars.

\ifpictures
\begin{figure}[!ht]
 \begin{center}
  \resizebox{11cm}{!}{\input{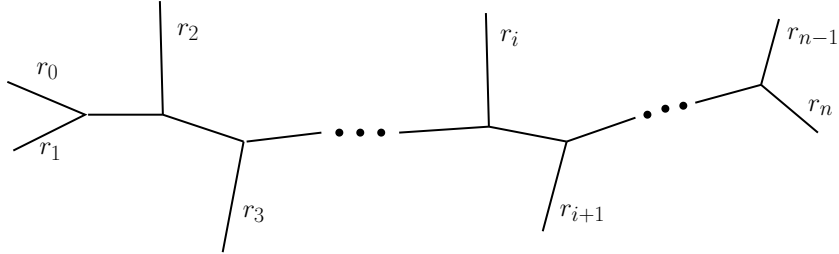}}
 \end{center}
\caption{\label{fig:caterpillar} A caterpillar line $L_n$ in $\R^n$.}
\end{figure}
\fi

\begin{remark}
\label{re:caterpillarcomplete} We remark that
tropical caterpillar lines in $\R^n$ can be written
as a complete intersection of the form $L = \bigcap_{i=1}^{n-1} \mathcal{T}(f_i)$
with linear polynomials $f_1, \ldots, f_{n-1}$. Namely, the
following representation of a caterpillar line $L$
in $\R^5$ with vertices $(0,0,0,0,0)$, $(-1,-1,0,$ $0,0)$,
$(-2,-2,-1,0,0)$,$(-3,-3,-2,-1,0)$ as a complete intersection
$L = \bigcap_{i=1}^{4} \mathcal{T}(f_i)$ generalizes to $\R^n$:
$\trop f_1 = 0\cdot x_1 \oplus 0\cdot x_2 \oplus 0 \cdot x_3 \oplus 1 \cdot x_4$,
$\trop f_2 = 1\cdot x_2 \oplus 0 \cdot x_3 \oplus 0 \cdot x_4$,
$\trop f_3 = 1\cdot x_3 \oplus 0\cdot x_4 \oplus 0 \cdot x_5$,
$\trop f_4 = 1\cdot x_4 \oplus 0\cdot x_5 \oplus 0$, where (for the sake of easier
reading) $\cdot$ denotes tropical multiplication.
We do not know if lines of other combinatorial types can always be written 
as a complete intersection of $n-1$ tropical hyperplanes.
\end{remark}

We describe the following constructions with many self-intersection points
(which can be regarded as lower bounds to the maximum number of
self-intersection points).
The proofs will be given in Sections~\ref{se:lowerbounds}
and \ref{se:upperbounds}.

\begin{theo} \label{theo:lower} For $n\geq 3$ we have:
\begin{enumerate}
\item[(a)] There exist a tropical line $L_n\subseteq\RR^n$ and a rational projection $\pi:\RR^n\to\RR^2$ such that $\pi(L_n)$ has 
\[\sum_{i=1}^{n-2}i \ = \ \left(\atopfrac{n-1}{2}\right)\] self-intersection points. 
\item[(b)] There exists a tropical curve $\mathcal{C}\subseteq\RR^n$ which is a transversal intersection 
of $n-1$ tropical hypersurfaces of degrees $d_1,\ldots,d_{n-1}$ and a rational projection $\pi:\RR^n\to\RR^2$ 
such that $\pi(\mathcal{C})$ has at least \[(d_1 \cdots d_{n-1})^2\cdot \left(\atopfrac{n-1}{2}\right) \] self-intersection points.
\end{enumerate}
\end{theo}

As an explicit upper bound, we show the following theorem
on the (unweighted) number of self-intersection points for 
caterpillar lines.

\begin{theo} \label{theo:upper} The image of a tropical line $L_n$ in $\RR^n$ which is a caterpillar can have at most $\sum_{i=1}^{n-2}i\ =\ \left(\atopfrac{n-1}{2}\right)$ self-intersection points. This bound is tight.
\end{theo}

By our earlier remark, in dimensions 3 and 4 this theorem covers 
all (non-degenerate) lines.
We conjecture that the upper bound in Theorem~\ref{theo:upper}
also holds for non-caterpillar lines in general dimension.

\begin{conj} \label{conj:upper} The image of a tropical line $L_n$ in $\RR^n$ can have at most 
$\sum_{i=1}^{n-2}i\ =\ \left(\atopfrac{n-1}{2}\right)$ self-intersection points.
\end{conj}

\subsection{Constructions with many 
self-intersection points\label{se:lowerbounds}}

In order to prove the first part of Theorem~\ref{theo:lower}, we start with 
the special case $n=3$. Then the general assertion will be
proven inductively.

\begin{bsp} \label{ex:lower3}
We consider in detail the case of a tropical line in $\RR^3$.
Let $\pi:\RR^3\to\RR^2$ be a rational projection, $M$ be
a matrix representing $\pi$, and $v = (v_1,v_2,v_3)^T$ be a vector
spanning the kernel of $\pi$.
In the case $v_1 = 0$ a projection with kernel generated by $(0,v_2,v_3)$ 
is degenerate.
If $v_1\not=0$ then $\pi$ can be described by the matrix
\[M=\left(\begin{array}{ccc}x & 1 & 0\\ y & 0 & 1\end{array}\right)\]
with $x,y\in \QQ$.
Let $L$ be a tropical line in $\RR^3$ of type $[12,34]$, i.e. a line with vertices
\[(p_1,p_2,p_3),\ (p_1+a,p_2+a,p_3) \, . \]
For simplicity, we consider the situation $a=1$ and $p_i=0$, $i \in \{1,2,3\}$.
There are four combinatorial possibilities for an intersection 
in the image of $\pi$: $\{1,3\}$, $\{1,4\}$, $\{2,3\}$, $\{2,4\}$.
A straightforward computation shows that, say, the
rays in directions $e^{(2)}$ and $e^{(3)}$ intersect in their interiors
if and only
$x < -1$ and $y > 0$.
In particular, there is a tropical line and a projection with one
self-intersection point.
\end{bsp}

In order to prove the first part of Theorem \ref{theo:lower} for general dimension,
we show the following stronger result. 
Let $L_n$ be a tropical caterpillar line in $\R^n$.
Let $r_i$ be the half-ray of $L_n$ emanating in direction $e^{(i+1)}$,
$0 \le i \le n-1$, and $r_n$ be the half-ray emanating in direction
$- \sum_{i=1}^n e^{(i)}$.
Without loss of generality we can assume $r_0 \cap r_1 \neq \emptyset$
and $r_{n-1} \cap r_{n} \neq \emptyset$ (see Figure~\ref{fig:caterpillar}).

\begin{lemma}
There is a projection $\pi:\RR^n\to \RR^2$ such that each ray $r_i$,
$2 \le i \le n-1$, intersects in the image with the ray $r_1$, 
such that the intersection point $p_i := \pi(r_i)\cap\pi(r_1)$ 
lies between $p_{i-1}$ and $p_{i+1}$ for $i>2$ and all 
images $\pi(r_i)$, $2 \le i \le n-1$, do not intersect with the images of 
the bounded edges.
\end{lemma}

\begin{proof}
The proof is by induction, where the case $n=3$ is clear from Example~\ref{ex:lower3}.
Let now $n+1$ be arbitrary.
Map the line with the projection $\sigma$ omitting the last
coordinate,
\[\sigma: \RR^{n+1}\to\RR^{n},\ (x_1, \ldots, x_{n+1}) \mapsto
  (x_1, \ldots, x_n) \, .
\]
Then $L_{n+1}$ is mapped to the nondegenerate line $L_n$. So by assumption there is a projection $\pi':\RR^{n}\to\RR^2$ satisfying the assertion of
the Lemma in dimension~$n$. 
Then the composition $\pi' \circ \sigma:\RR^{n+1}\to\RR^2$ maps $r_{n+1}$ to a point on $\pi'(r_n)$. The corresponding matrix of $\pi' \circ \sigma$ 
has the form 
\[A_{\pi' \circ \sigma}=\left(\begin{matrix}a'_{11} & \ldots & a'_{1 n+1} & 0\\ a'_{21} & \ldots & a'_{2 n+1} & 0\end{matrix}\right).\] 
Figure \ref{piLm-1} shows an example for $\pi'(L_{n})$ for $n=4$.

\ifpictures
\begin{figure}[ht!]\setlength{\unitlength}{0.7cm}
\begin{center}
 \resizebox{7cm}{!}{
 \input{Bilder/fig21.pstex_t}}
\end{center}\caption{\label{piLm-1} $\pi'(L_{4})$}
\end{figure}
\fi

Let $\pi:\RR^{n+1}\to\RR^2$ be defined by a matrix with the same columns as $A_{\pi' \circ \sigma}$ except the last one. 
Due to the balancing condition we can choose 
the image of the coordinate vector
$e^{(n+1)}$ such that the ray $r_{n}$ is mapped to a 
ray which (in the two-dimensional picture) lies below the 
images of the bounded edges and has an intersection point $p_{n}$ with 
$\pi(r_1)=\pi' \circ \sigma(r_1)$ lying above $p_{n-1}$.  
By the induction assumption $\pi(r_{n})$ intersects with all 
$\pi(r_i), 1\leq i\leq n-1$.
So $\pi(r_{n})$ has $n-1$ self-intersection points. Altogether there are \[\sum_{i=1}^{n-2}i+(n-1)\ =\ \sum_{i=1}^{n-1}i\] intersection points under $\pi$.
\end{proof}

In order to show the second part of Theorem~\ref{theo:lower}, we observe
that by Remark~\ref{re:caterpillarcomplete}
the tropical line $L \subseteq \R^n$ from part a) (which was a
caterpillar line) can be written as 
a complete intersection of the form $L \ = \ \bigcap_{i=1}^{n-1} \T(f_i)$
with linear polynomials $f_1, \ldots, f_{n-1}$.
Pick such a line $L$ from the first part having $\binom{n-1}{2}$ self-intersection points
under $\pi$ and note that in that construction all intersection points
occur on unbounded
rays. Let $g_i$ be the product of $n$ perturbed copies of $f_i$. Since in this
way every unbounded ray is transformed into $d$ copies,
every self-intersection point is transformed into $d_1^2 \cdots d_{n-1}^2$ copies.
Hence, this gives us $(d_1 \cdots d_{n-1})^2 \binom{n-1}{2}$ self-intersection
points (and there could be more, stemming from intersections involving
unbounded edges).

\begin{bsp}\label{ex:resultant}
Let $K=\C\{\{t\}\}$ the field of Puiseux series with the natural valuation,
and let 
\begin{eqnarray*}
f_1 & = & (t^{\epsilon_1}x+t^{\epsilon_2}y+t^{1+\epsilon_3}z+t^{3+\epsilon_4})\cdot (x+y+tz+t^3)\\
f_2 & = & (t^{1+\delta_1}x+t^{\delta_2}y+t^{\delta_3}z+t^{\delta_4})\cdot (tx+y+z+1)\\
\end{eqnarray*}
with 
$\varepsilon_1 := \frac{1}{1000}$, 
$\varepsilon_2 := \frac{3}{1000}$, 
$\varepsilon_3 := \frac{5}{1000}$,
$\varepsilon_4 := \frac{7}{1000}$, 
$\delta_1 := \frac{11}{1000}$, 
$\delta_2 := \frac{13}{1000}$, 
$\delta_3 := \frac{17}{1000}$, 
$\delta_4 := \frac{1}{1000}$.

So each tropical variety $\T(f_i)$ is the union of two tropical hyperplanes
and therefore $d_1=d_2=2$. The intersection $\T(f_1)\cap\T(f_2)$ is a 
tropical curve with four unbounded rays in each of the directions $e^{(1)},e^{(2)},e^{(3)},-e^{(2)}-e^{(2)}-e^{(3)}$. It is the union of four tropical lines, for example 
\begin{equation}
\label{lline}\T(t^{\epsilon_1}x+t^{\epsilon_2}y+t^{1+\epsilon_3}z+t^{3+\epsilon_4})\ \cap \T(t^{1+\delta_1}x+t^{\delta_2}y+t^{\delta_3}z+t^{\delta_4}) \, .
\end{equation}
Under the projection 
\[\pi:\RR^3\to\RR^2,\ \; x\mapsto\left(\begin{matrix}1& 0& 1\\0& 1& 2\end{matrix}\right) \]
the image of the tropical line resulting from $\varepsilon_i=\delta_i=0$ 
has one self-intersection point. Our theorem guarantees us at least $(2 \cdot 2)^2 = 16$
self-intersection points. In fact, there are actually 28 self-intersection points,
as can be seen from the induced subdivision of the corresponding fiber polytope
under a monomial map as in Example~\ref{ex:monomialmap},
depicted on the left side of Figure~\ref{fig:resultant}.

\begin{figure}[!ht]
\begin{center}
\includegraphics[height=5cm]{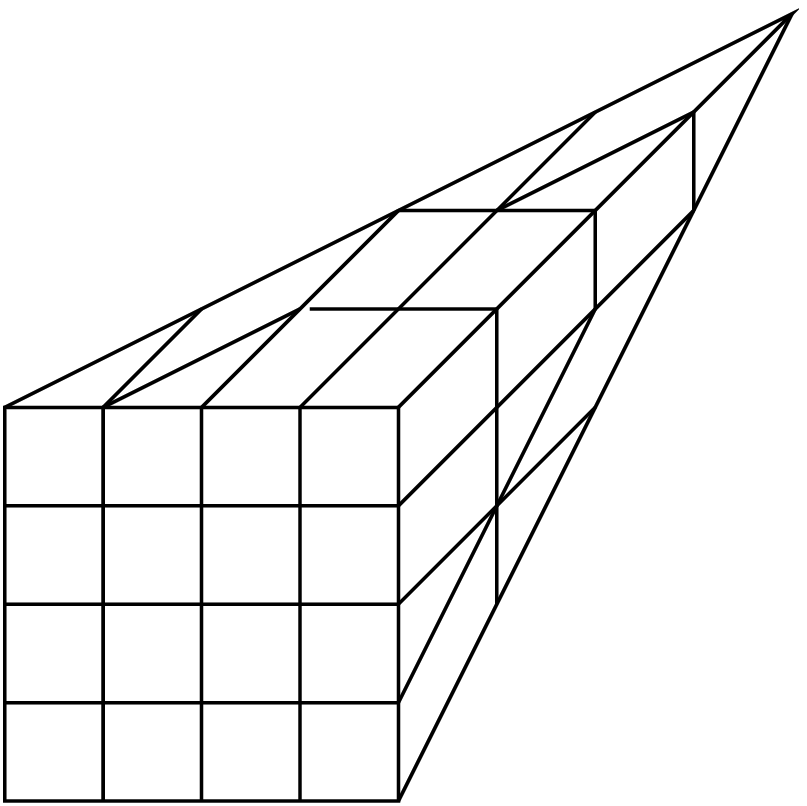} \hspace*{1cm}
\includegraphics[height=5cm]{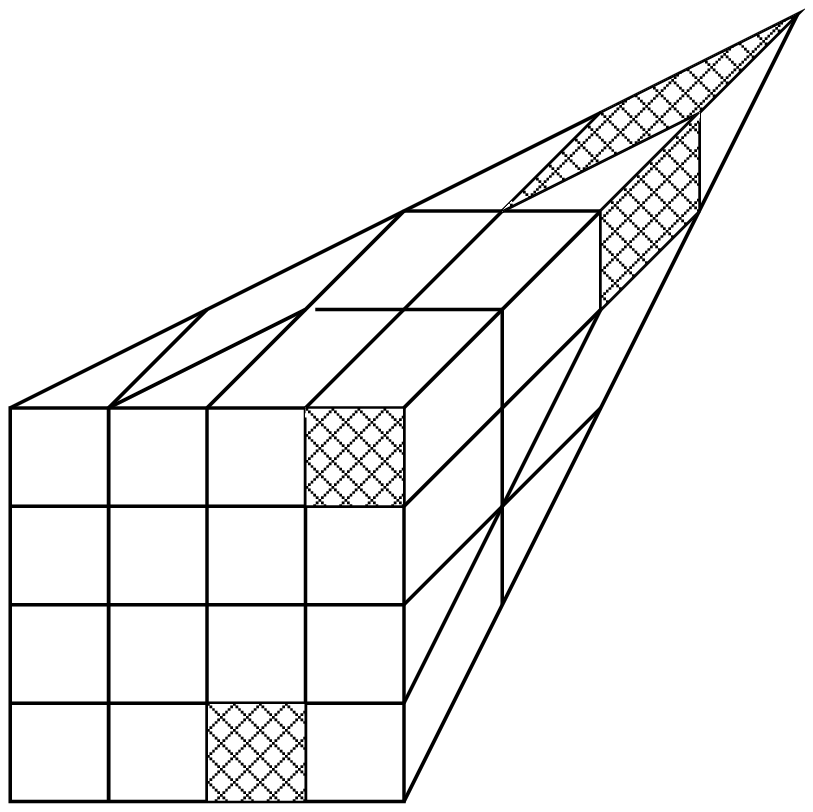}

\end{center}
\caption{\label{fig:resultant}The subdivided Newton polytope of the polynomial generating the hypersurface $\pi^{-1}\pi(\T(I))$.}
\end{figure}

Here the projection of (\ref{lline}) intersects the projection of the line 
\[
\T(x+y+tz+t^3) \cap \T(t^{1+\delta_1}x+t^{\delta_2}y+t^{\delta_3}z+t^{\delta_4})
\]
in four (out of the 28) points, two of them belonging to the intersection of the unbounded rays, the other two are intersections of a bounded edge with an unbounded ray; see the right picture of figure \ref{fig:resultant} for the corresponding dual cells in the subdivided Newton polytope.
\end{bsp}

\subsection{Upper bounds for the number of self-intersection points\label{se:upperbounds}}

\begin{bsp}
Let $L_4$ be a tropical line in $\R^4$.
There exist two pairs of half-rays having a non-empty intersection. 
We divide $L_4$ in three parts.
Let $R$ and $G$ be the images of these two pairs of half-rays, respectively,
and let $B$ be the image of the remaining segments and half-rays,
see Figure~\ref{line}.

\begin{figure}[!ht]
\begin{center}
\setlength{\unitlength}{0.4cm}
\includegraphics[height=4cm]{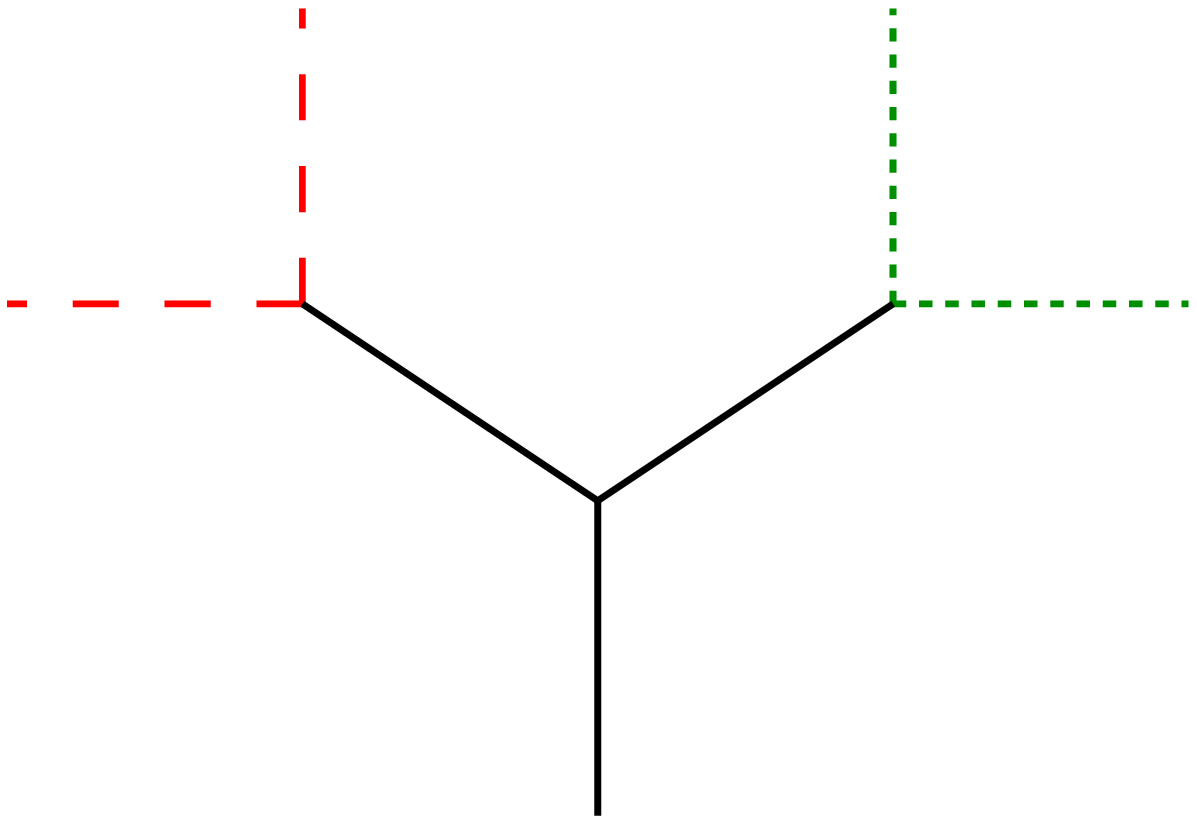}\hspace*{1cm}
\includegraphics[height=5cm]{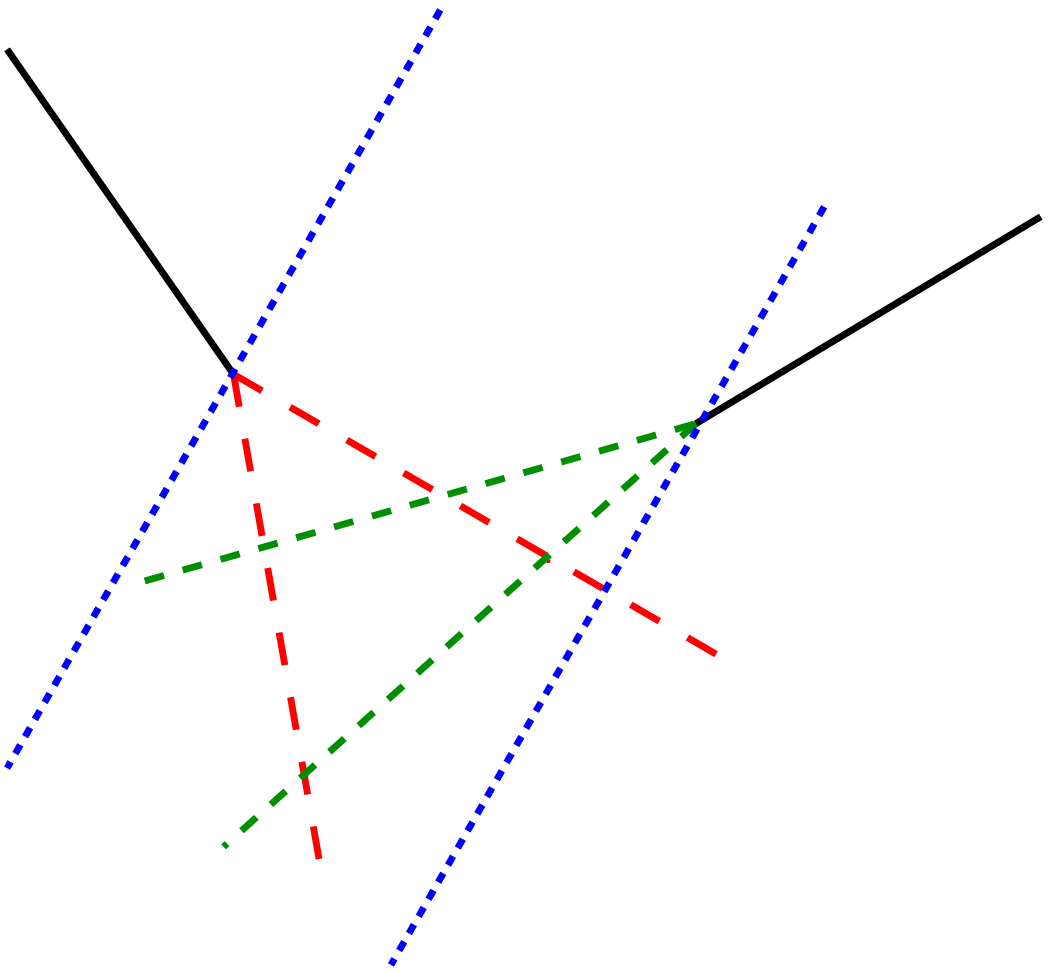}

\caption{\label{line} The first picture shows a tropical line $L$ in $\R^4$. The two
pairs of half-rays with a non-empty intersection are drawn in
dashed (red) and dotted (green) lines, respectively. 
The remaining segments and half-ray are drawn in solid lines. The second picture shows that there cannot be four selfintersection points.}
\end{center}
\end{figure}

If neither $R$ nor $G$ intersects with $B$ then by the total concavity
for plane tropical curves (as defined in section~\ref{se:tropical}), 
there can be at most three 
intersections between $B$ and $G$: If there are four selfintersection points then the two black bounded edges lie in two different nonintersecting halfspaces but have to intersect in one point. This is a contradiction.\\
If exactly one of $R$ and $G$ intersect with $B$, say $G$, then
by the total concavity $R$ can intersect $G$ in at most two
points.
If both $R$ and $G$ intersect with $B$ then by the total concavity
$R$ and $G$ can intersect in at most one point.

Hence, the image of a tropical line $L_3\subseteq\RR^3$ can have at most one 
and the image of a tropical line $L_4\subseteq\RR^4$ can have at most three self-intersection points. 
\end{bsp}

\ifpictures
\begin{figure}[!ht]
 \begin{center}
  \resizebox{11cm}{!}{\input{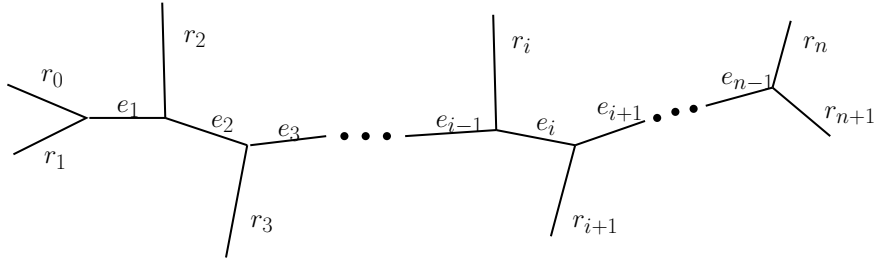}}
 \end{center}
\caption{\label{fig:caterpillar2} The caterpillar line $L_{n+1}$ in the proof
 of Theorem~\ref{theo:upper}}
\end{figure}
\fi

\medskip

\noindent
\emph{Proof of Theorem~\ref{theo:upper}.}
For $n=3,4$ we have seen the assertion. Now assume the assertion is 
true for some $n\geq 3$, and show it inductively for $n+1$.
Let $L_{n+1} \subseteq \R^{n+1}$ be a tropical caterpillar line and 
$\pi: \R^{n+1} \to \R^2$, $x \mapsto Ax$ be a rational projection.
The unbounded rays of $L_{n+1}$ are denoted by $r_0, \ldots, r_{n+1}$,
the bounded segments by $e_1, \ldots, e_{n-1}$ (see Figure~\ref{fig:caterpillar2}).

Let $v$ be the vertex incident to $r_0$, $r_1$ and $e_1$, and let
$w$ be the vertex incident to $e_1$, $e_2$ and $r_2$.
Consider the polyhedral complex in $\R^2$ obtained by 
replacing the projections $\pi(r_0)$, $\pi(r_1)$ and $\pi(e_1)$ 
by a ray emanating from $\pi(w)$ into the direction of $\pi(e_1)$.
This polyhedral complex
is the projection of some tropical line $L_n \subseteq \R^n$.
We denote that projection by $\pi'$ and by $\tilde{r}$ the ray in 
$L_n$ projecting to the new ray emanating from $\pi(w)$.
See Figure~\ref{Reduktion}.

\ifpictures
\begin{figure}[!ht]
\begin{center}
\resizebox{5cm}{!}{\input{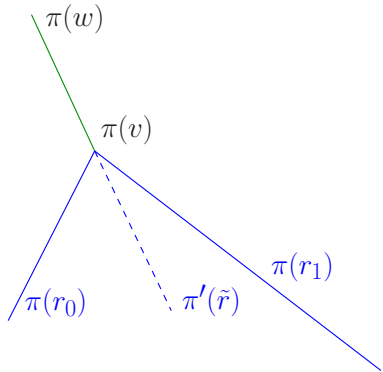}}
\caption{\label{Reduktion} The rays $\pi(r_0)$, $\pi(r_1)$ and
$\pi'(\tilde{r})$ in $\R^2$.}
\end{center}
\end{figure}
\fi

By the induction hypothesis, $\pi'(L_n)$ has at most $\binom{n-1}{2}$
self-intersection points. Denote by $\SIP_{\pi(r_i)}$ the 
self-intersection points of $\pi(r_i)$ with $\pi(L_{n+1})$ where
we do not count the (trivial)
intersection points with segments and half-rays emanating from $\pi(v)$,
$1 \le i \le 2$.
And analogously, let $\SIP_{\pi'(\tilde{r})}$ be the self-intersection points of 
$\pi'(\tilde{r})$ with $\pi(L_n)$.
In order to complete the inductive proof, we have to show
\[
  \sharp \SIP_{\pi(r_0)}+ \sharp \SIP_{\pi(r_1)}-\sharp \SIP_{\pi'(\tilde{r})} \ \leq \ n-1 \, .
\]
Each line $L_{n+1}$ has $n-1$ segments. One of them emanates from $v$.
Each other segment contributes to the above sum with at most 1 because its image cannot intersect with $\pi(r_0)$ and $\pi(r_1)$ and not with $\pi'(\tilde{r})$.

If no segment of $L_{n+1}$ has contribution $1$, then by the
concavity condition
there are at most $n-1$ rays with contribution $1$ 
(see Figure \ref{fig:case1}),
and we are done.
Assume in the following that there exists at least one segment with
contribution~1.

\ifpictures
\begin{figure}[t]
 \begin{center}
  \resizebox{5cm}{!}{
 \input{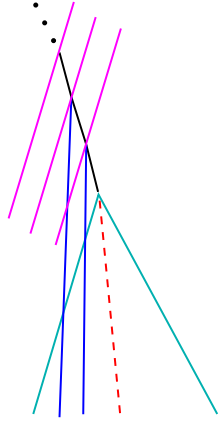}}
 \end{center}
\caption{\label{fig:case1} Not all rays can have contribution 1}
\end{figure}
\fi

For a segment $c=e_i$ or a half-ray $c=r_i$ define the \emph{contribution} $\alpha(c)$ by  
\[\contr(c)\ := \ \delta_{\pi(c)\cap \pi(r_0)}+\delta_{\pi(c)\cap \pi(r_1)}-\delta_{\pi(c)\cap\pi'(\tilde{r})} \, ,
\]
where $\delta_{a \cap b}$ is defined to be 1 if and only if the interiors
of $a$ and $b$ intersect, and 0 otherwise.
Then 
\begin{equation}
\label{eq:sumofcontributions}
\sum_{\atopfrac{c \mbox{ \tiny an edge of } L_{n+1}}{c\not\in \{r_0,r_1,e_1\}}}\hspace*{-0.7cm}\contr(c)\ = \ \sharp \SIP_{\pi(r_0)}+ \sharp \SIP_{\pi(r_1)}-\sharp \SIP_{\pi'(\tilde{r})} \, .
\end{equation}
Let $\mathcal{R}$ be the subset of rays $r_i$ in $\{r_2, \ldots, r_{n+1}\}$
which satisfy $\alpha(r_i) = 1$.
We will construct a map $\alpha$ from $\mathcal{R}$ 
to the set $\mathcal{B}$ of segments with nonpositive contribution.
This map will not be injective, but will satisfy the following condition.
Any segment $e$ with $\alpha(e) = 0$ will have at most one
preimage. Any segment $e$ with $\alpha(e) = -1$ will have at
most two preimages.
The existence of the map then implies that when passing over from
dimension $n$ to dimension $n+1$ the sum of the 
contributions~\eqref{eq:sumofcontributions} is increased by at most
the number of segments of $L_{n+1}$, i.e., by at most $n-1$,
as desired.

In order to construct the map, we proceed along the caterpillar
line. Denote by $e_{i_1},\ldots,e_{i_r}$ with $i_1<i_2<\dots<i_r$ 
the segments with $\contr(e_{i_j})\not=0$. Note that $i_1 \ge 2$.
We call the sequence $(e_{i_1}, \ldots, e_{i_r})$ the sequence of
\emph{separating} edges.
These separating edges naturally divide the caterpillar lines into
several parts each of which will be treated separately. 
We construct the map as follows.

\medskip

\noindent
{\tt Case 1:} $2\leq i\leq i_1$.
Each $r_i$ with positive contribution is assigned to 
the adjacent edge before $r_i$, 
\[e(r_i) \ = \ e_{i-1}\mbox{ for }\contr(r_i) \ = \ 1 \, . \]
Observe that by the total concavity property not all the rays $r_i$ with
$2\leq i\leq i_1$ can have $\contr(r_i)=1$.
Hence, there is a segment $e_l$ which has not been used yet
and thus can be used later (in case 3).

\medskip

\noindent
{\tt Case 2:} Let $e_s$ and $e_t$ be two separating edges which are
neighboring within the sequence of separating edges.
For $s+1 \le i \le t$ we assign as follows.
If $\contr(e_s)=-1$ then we assign
\[e(r_i) \ = \ e_{i-1} \mbox{ for } \contr(r_i) \ = \ 1 \, . \]
Note that $e_{t}$ is used at most twice. If
$\contr(e_{t})=1$ then we distinguish two cases.

\smallskip

\noindent
(a) If $\contr(e_{t})=1$ then, by similar arguments to case~1, there
exists a ray $r_k$ with $s+1\leq k\leq t$ and $\contr(r_k)\leq 0$. 
So, with the exception of $i=t$, 
we assign $r_i$ to the adjacent edge $e_i$ lying after $r_i$,
\[e(r_i) \ = \ e_i\mbox{ for } \contr(r_i) \ = \ 1 \, , \quad
s+1 \le i \le t - 1 \, . 
\]
If $\contr(r_{t})=1$ we use the (up to now unused) edge $e_{k}$
and assign $e(r_{t})=e_{k}$.

\smallskip

\noindent
(b) If $\contr(e_{t})=-1$ 
then we can assign again to the edge lying after $r_i$,
\[e(r_i) \ = \ e_i \mbox{ for } \contr(r_i) \ = \ 1 \, . \]

\medskip

\noindent
{\tt Case 3:} $i_r+1\leq i \leq n+1$. We distinguish two cases.

\smallskip

\noindent
(a) $\contr(e_{i_r})=1$. By the concavity condition,
not all $r_i,\ i_r+1\leq i\leq n+1$ can have $\contr(r_i)=1$.
Now we can assign
\[e(r_i) \ = \ e_i \mbox{ for } \contr(r_i) \ = \ 1 \, , \quad
  i_r+1\leq i\leq n-1 \, . \]
Since not all rays have contribution $1$ there is a ray $r_k,\ i_r+1\leq k\leq n+1$ with $\contr(r_k)\not=1$. So it remains to assign the last rays $r_n$ and $r_{n+1}$ to an appropriate edge if necessary. These two rays can be
assigned to $e_l$ and $e_k$, where $e_l$ is defined in case~1.

\smallskip

\noindent
(b) $\contr(e_{i_r})=-1$. Then we can assign
\[e(r_i) \ = \ e_{i-1} \mbox{ for } \contr(r_i) \ = \ 1 \, , 
  \quad i_r+1\leq i\leq n \]
and 
\[e(r_{n+1}) \ = \ e_l \mbox{ if } \contr(r_{n+1}) \ = \ 1, \mbox{ where $l$ was defined in case 1{\,}.} \]

\smallskip

Altogether, to each ray $r_i$ with $\contr(r_i)=1$ we have assigned 
a segment $e(r_i)$ satisfying the conditions stated above.
This proves the claim.
\hfill $\Box$

\subsection*{Acknowledgment.} We thank an anonymous referee for
very helpful remarks and corrections.

\bibliography{Literatur}
\bibliographystyle{plain}
\end{document}